\documentclass[a4paper,reqno,11pt]{amsart}
\usepackage{amsfonts,amsmath,amsthm,amssymb,stmaryrd}
\usepackage{mathrsfs}

\newtheorem{theorem}{Theorem}[section]
\newtheorem{lemma}[theorem]{Lemma}

\newtheorem{lem}{Lemma}[section]

\newtheorem{prop}[lem]{Proposition}
\newtheorem{thm}[lem]{Theorem}
\newtheorem{remark}[lem]{Remark}

\numberwithin{equation}{section}
\usepackage{color}
\allowdisplaybreaks

\usepackage[left=1 in, right=1 in,top=1 in, bottom=1 in]{geometry}

\providecommand{\abs}[1]{\left\vert#1\right\vert}
\providecommand{\norm}[1]{\left\Vert#1\right\Vert}

\providecommand{\sd}[1]{\mathcal{D}_{#1}}
\providecommand{\se}[1]{\mathcal{E}_{#1}}

\providecommand{\ns}[1]{\norm{#1}^2}

\DeclareMathOperator{\diverge}{div}

\providecommand{\norm}[1]{\left\Vert#1\right\Vert}
\def\ls{\lesssim}

\def\dt{\partial_t}

\def\pa{\partial}

\def\RRvert2{\right \vert\! \right\vert}
\def\Lvert3{\left \vert\!\left\vert\!\left\vert}
\def\Rvert3{\right \vert\!\right\vert\!\right\vert}

\def\nab{\nabla}
\def\al{\alpha}
\def\dt{\partial_t}
\def\dtt{ \frac{d}{dt}}
\def\hal{\frac{1}{2}}

\def\ls{\lesssim}
\def\p{\partial}

\def\da{\Delta_{\mathcal{A}}}
\def\naba{\nab_{\mathcal{A}}}
\def\diva{\diverge_{\mathcal{A}}}

\def\a{\mathcal{A}}
\def\f{\mathcal{F}_{2N}}

\def\fj1{\mathcal{J}^{-1}}

\def\x{\mathcal{X}}
\def\y{\mathcal{Y}}

\def\q{q}

\title[Viscous non-resistive MHD systems]{Global well-posedness of an initial-boundary value problem for viscous non-resistive MHD systems}

\author{Zhong Tan}
\address{
School of Mathematical Sciences\\
Xiamen University\\
Xiamen, Fujian 361005, China}
\email[Z. Tan]{ztan85@xmu.edu.cn}
\thanks{Z. Tan was supported by the National Natural Science Foundation of China (11531010)}

\author{Yanjin Wang}
\address{
School of Mathematical Sciences\\
Xiamen University\\
Xiamen, Fujian 361005, China}
\email[Y. J. Wang]{yanjin$\_$wang@xmu.edu.cn}
\thanks{Y. J. Wang was supported by the Fujian Province
Natural Science Funds for Distinguished Young Scholar (2015J06001).}

\subjclass[2010]{35Q30, 76D03, 76N10, 76W05}
\keywords{MHD; Global well-posedness; Compressible; Incompressible.}

\begin{document}

\begin{abstract}
This paper concerns the viscous and non-resistive MHD systems which govern the motion of electrically conducting fluids interacting with magnetic fields. We consider an initial-boundary value problem for both compressible and (nonhomogeneous and homogeneous) incompressible fluids in an infinite flat layer. We prove the global well-posedness of the systems around a uniform magnetic field which is vertical to the layer. Moreover, the solution converges to the steady state at an almost exponential rate as time goes to infinity. Our proof relies on a two-tier energy method for the reformulated systems in Lagrangian coordinates.
\end{abstract}

\maketitle


\section{Introduction}

\subsection{Formulation}

The dynamics of electrically conducting fluids interacting with magnetic fields can be described by the equations of magnetohydrodynamics (MHD) \cite{Ca,Co,LL}. In this paper, we are concerned with the global existence of smooth solutions for the MHD systems with taking into account the viscosity and neglecting the resistivity. For the compressible flow, the MHD system takes the following form:
\begin{equation}\label{ns_euler0c}
\begin{cases}
 \dt \tilde{\rho} +\diverge( \tilde{\rho}\tilde{u})=0& \text{in } \Omega\\
\tilde {\rho}   (\partial_t\tilde{u}  +   \tilde{u}\cdot \nabla \tilde{u} ) -\mu\Delta \tilde{u}-(\mu+\mu')\nabla\diverge \tilde u+\nabla \left(    P (\tilde{\rho} )+\frac{\kappa}{2}|\tilde{B}|^2 \right) =\kappa\tilde{B} \cdot \nabla \tilde{B}  & \text{in } \Omega
\\ \partial_t\tilde{B}  +   \tilde{u} \cdot \nabla \tilde{B}  -\tilde{B} \cdot \nabla \tilde{u}+\tilde{B} \diverge \tilde u=0 & \text{in } \Omega
\\\diverge{\tilde B}=0 &\text{in }\Omega
\\ (\tilde{\rho},\tilde u,   \tilde B)\mid_{t=0}=(\tilde{\rho}_0,\tilde u_0,  \tilde B_0).
\end{cases}
\end{equation}
 Here $\tilde{\rho}(t,y), \tilde u(t,y) $ and $\tilde B(t,y)$ denotes the density, velocity and magnetic field functions, respectively, where time $t\in \mathbb{R}^+$ and position $y\in \Omega$  with $\Omega$ the domain occupied by the fluid.
The pressure $\tilde P  =  P (\tilde{\rho} )>0$ is a function of the density, which is assumed to be smooth and strictly increasing. $\mu $  and  $\mu' $ are the viscosity coefficients satisfying the physical conditions
\begin{equation}\label{visco}
 \mu>0 \text{ and }\mu'+ \frac{2}{3}\mu\ge 0,
\end{equation}
and $\kappa>0$ is the permeability coefficient. If the fluid is incompressible, then the velocity is divergence free and the pressure becomes a new unknown $\tilde p(t,y)$; the MHD system takes the following form:
\begin{equation}\label{ns_euler0ic}
\begin{cases}
 \dt \tilde{\rho} +\tilde{u}\cdot\nabla\tilde{\rho}=0& \text{in } \Omega\\
\tilde {\rho}   (\partial_t\tilde{u}  +   \tilde{u}\cdot \nabla \tilde{u} ) -\mu\Delta \tilde{u}+\nabla \left(\tilde p+\frac{\kappa}{2}|\tilde{B}|^2 \right) =\kappa\tilde{B} \cdot \nabla \tilde{B} & \text{in } \Omega
\\ \partial_t\tilde{B}  +   \tilde{u} \cdot \nabla \tilde{B} -\tilde{B} \cdot \nabla \tilde{u}=0 & \text{in } \Omega
\\\diverge{\tilde u}=\diverge{\tilde B}=0 &\text{in }\Omega
\\ (\tilde{\rho},\tilde u,   \tilde B)\mid_{t=0}=(\tilde{\rho}_0,\tilde u_0,  \tilde B_0).
\end{cases}
\end{equation}
If $\tilde\rho\equiv 1$, then the system \eqref{ns_euler0ic} reduces to the homogeneous one:
\begin{equation}\label{ns_euler0ich}
\begin{cases}
 \partial_t\tilde{u}  +   \tilde{u}\cdot \nabla \tilde{u} -\mu\Delta \tilde{u}+\nabla \left(\tilde p+\frac{\kappa}{2}|\tilde{B}|^2 \right) =\kappa\tilde{B} \cdot \nabla \tilde{B} & \text{in } \Omega
\\ \partial_t\tilde{B}  +   \tilde{u} \cdot \nabla \tilde{B} -\tilde{B} \cdot \nabla \tilde{u}=0 & \text{in } \Omega
\\\diverge{\tilde u}=\diverge{\tilde B}=0 &\text{in }\Omega
\\ ( \tilde u,   \tilde B)\mid_{t=0}=( \tilde u_0,  \tilde B_0).
\end{cases}
\end{equation}

The main difficulty of studying these MHD systems lies in the non-resistivity of the magnetic equation. It is classical that the viscous and resistive homogeneous MHD system has a unique global classical solution \cite{DL,ST}, at least for the small initial data. It is extremely interesting that the inviscid and non-resistive homogeneous MHD system also poses a unique global classical solution around a nonzero uniform magnetic field \cite{BSS}. It is then natural to ask whether the MHD systems with only the viscosity or resistivity admit global classical solutions or develop singularity in finite time. The inviscid and resistive homogeneous 2D MHD system has a global weak solution in $H^1$, but the question if such weak solutions are unique or can be improved to be global classical solutions remains open \cite{CW,LZH}. The global existence of classical solutions to the viscous and non-resistive homogeneous MHD system \eqref{ns_euler0ich} is established only recently around a nonzero uniform magnetic field; we refer to \cite{LXZ,RWXZ,Z,HL} for the 2D case and \cite{XZ,AZ} for the 3D case, and also \cite{LZ,LZT} for a 3D MHD-type system. For the viscous and non-resistive compressible MHD system \eqref{ns_euler0c} the global existence of classical solutions is established recently in \cite{H} for the 2D case. We remark that the analysis in \cite{RWXZ,Z,HL,H} for the 2D case exploited greatly the condition $\diverge \tilde B=0$, while \cite{LXZ,XZ} employed the Lagrangian reformulation of the problem and required the initial magnetic field, $\tilde B_0$, satisfy the following admissible condition:
\begin{equation}
\int_{\mathbb{R}}(\tilde B_0 - e_n)(Z(t,\alpha))dt=0\text{ for all }\alpha\in \mathbb{R}^{n-1}\times\{0\},\ n=2,3
\end{equation}
with $Z(t,\alpha)$ being determined by
\begin{equation} 
\begin{cases}
\displaystyle	\frac{dZ(t,\al)}{dt}=\tilde B(Z(t,\al))
	\\ Z(0,\al)=\al,
\end{cases}
\end{equation}
and such condition was removed in \cite{AZ}.

However, all these global well-posedness results of \eqref{ns_euler0ich} only consider the Cauchy problem;  some of the techniques such as the anisotropic Littlewood-Paley analysis and some crucial integration by parts in spatial variables employed in these papers can not be applied directly to the initial-boundary value problem. In this paper, we prove the global existence of smooth solutions to the systems \eqref{ns_euler0c} and \eqref{ns_euler0ic} in the horizontally infinite flat layer $\Omega=\mathbb{R}^2\times(0,1)$. We impose the usual no-slip condition on the boundary:
\begin{equation}\label{bounder}
\tilde u=0\quad \text{on }\p\Omega:=\mathbb{R}^2\times \{0,1\}.
\end{equation}
Note that the continuity equation and the magnetic equation are hyperbolic and characteristic, and hence no boundary condition needs to be imposed for the density and the magnetic field.

As in \cite{LXZ,XZ,W}, it is more convenient for us to reformulate the systems by using Lagrangian coordinates so as to capture the weak dissipation of the magnetic field. To this end, we assume that there exists an invertible mapping $
\eta_{0} :\Omega \rightarrow
\Omega
$
so that $\partial\Omega=\eta_{0}(\partial\Omega)
$. Define the flow map
$\eta$ as the solution to
\begin{equation}
\begin{cases}
\partial_t\eta(t,x)=\tilde u(t,\eta(t,x))
\\
\eta(0,x)=\eta_{0}(x).
\end{cases}
\end{equation}
We think of Eulerian coordinates as $(t,y)\in \mathbb{R}^+\times\Omega$ with $y =
\eta(t,x)$, whereas we think of Lagrangian coordinates as
$(t,x)\in \mathbb{R}^+\times\Omega$. In order to switch back and
forth from Lagrangian to Eulerian coordinates we assume that $\eta(t,\cdot)$ are invertible and that $\partial\Omega=\eta(t,\partial\Omega)
$ (which follows by $\partial\Omega=\eta_{0}(\partial\Omega)
$ and $\tilde u=0$ on $\p\Omega$).

If $\eta-Id $ is sufficiently small (in an appropriate Sobolev space), then the mapping $\eta$ is a diffeomorphism. For the compressible fluid, we define the Lagrangian unknowns:
\begin{equation}
(\rho,u,B)(t,x)=(\tilde \rho,\tilde u,\tilde B)(t,\eta(t,x)),\quad (t,x)\in \mathbb{R}^+\times\Omega.
\end{equation}
Then the system \eqref{ns_euler0c} and \eqref{bounder} becomes the following system for $(\eta,\rho,u,B)$:
\begin{equation}\label{lagrangianc}
\begin{cases}
\partial_t \eta = u  &
 \text{in }
\Omega
\\
\partial_t \rho +{\rho}\diverge_\a    u =0 & \text{in }
\Omega
 \\ \rho  \partial_t    u   -\mu\Delta_\a u-(\mu+\mu')\nabla_\a \diverge_\a u +\nabla_\a   \left(P(\rho)+\frac{\kappa}{2}| {B}|^2 \right)  =\kappa {B} \cdot \nabla_\a  {B}    & \text{in }
\Omega
 \\ \partial_t B-B\cdot\nabla_\a u+B \diverge_\a u=0& \text{in }
\Omega
\\  \diva B=0& \text{in }
\Omega  \\
 {u}=0 &\text{on }\p\Omega
\\ (\eta,\rho,  u,   B)\mid_{t=0}=( \eta_0,\rho_0, u_0,    B_0).
\end{cases}
\end{equation}
Here $\a=((\nabla\eta)^{-1})^T$ and we have written the differential operators $\nabla_\a, \diverge_\a$, and $\Delta_\a$ with their actions given by
$(\naba f)_i := \a_{ij} \p_j f$, $\diva X := \a_{ij}\p_j X_i $ and $\da f := \diva \naba f$
for appropriate $f$ and $X$. For the incompressible fluid, we define the Lagrangian unknowns:
\begin{equation}
(\rho,u,p,B)(t,x)=(\tilde \rho,\tilde u,\tilde p+\frac{\kappa}{2}|\tilde{B}|^2,\tilde B)(t,\eta(t,x)),\quad (t,x)\in \mathbb{R}^+\times\Omega.
\end{equation}
Note that the continuity equation becomes $\dt\rho=0$, i.e., $\rho(t,x)\equiv   \rho_0(x)(=\tilde \rho_0(\eta_0(x)))$, and hence the density can be regarded as a parameter function in Lagrangian coordinates. Then the system \eqref{ns_euler0ic} and \eqref{bounder} becomes the following system for $(\eta,u,p,B)$:
\begin{equation}\label{lagrangianic}
\begin{cases}
\partial_t\eta=u & \text{in }
\Omega
 \\  \rho_0\partial_t u - \mu\da u+\naba p=\kappa B\cdot\nabla_\a B & \text{in }
\Omega
 \\ \partial_t B-B\cdot\nabla_\a u=0& \text{in }
\Omega
\\ \diva u=\diva B=0& \text{in }
\Omega \\
 {u}=0 &\text{on }\p\Omega
\\ (\eta,  u,   B)\mid_{t=0}=( \eta_0, u_0,    B_0).
\end{cases}
\end{equation}

We now turn to study the equivalently reformulated systems \eqref{lagrangianc} and \eqref{lagrangianic} in Lagrangian coordinates. For the uniform vertical magnetic field $\bar B=(0,0,\bar b)$ with $\bar b\neq 0$, we will prove the global existence of smooth solutions of \eqref{lagrangianc} around the steady state $(\eta,\rho,  u,   B)=(Id,\bar\rho, 0,\bar B)$ with the uniform density $\bar\rho>0$  and \eqref{lagrangianic} around the steady state $(\eta,   u,   B)=(Id,  0,\bar B)$. Our results show that under certain necessary conditions on the initial data, these systems admit a global unique smooth solution for the sufficiently small initial perturbation; moreover, the solution converges to the steady state at an almost exponential rate as time goes to infinity. The global well-posedness and decay of the original systems \eqref{ns_euler0c} and \eqref{ns_euler0ic} with the boundary condition \eqref{bounder} follow by the change of variable, correspondingly.

\subsection{Conserved quantities}

In our global well-posedness of \eqref{lagrangianc} and \eqref{lagrangianic}, it is a very key to find out the conserved quantities. These conserved quantities are known to be very important for the global well-posedness results, see for instance \cite{LXZ,XZ,HL,H,AZ}. They indicate the conditions needed to be imposed on the initial data if we want to show the convergence of the solutions towards the steady states as time goes to infinity in our functional framework.

We first deal with the compressible MHD system \eqref{lagrangianc}. We denote $J={\rm det} (\nabla\eta)$, the Jacobian of the coordinate transformation. First, direct computation yields that
\begin{equation}\label{jequ}
\partial_t J =J\diverge_\a    u ,
\end{equation}
which together with the continuity equation implies
\begin{equation}
\dt(\rho J)=0.
 \end{equation}
Second, applying $J\a^T $ to the magnetic equation, by \eqref{jequ}, we obtain
\begin{align} 
J\a_{ji}\partial_tB_j&=J\a_{ji}B_k\a_{kl}\partial_t(\partial_l\eta_j)-J\a_{ji}B_j \diverge_\a u\nonumber
\\&=-J \partial_t\a_{ji}B_k\a_{kl}\partial_l\eta_j-\dt J\a_{ji}B_j=-JB_j\partial_t\a_{ji}-\dt J\a_{ji}B_j. 
\end{align}
This implies that
\begin{equation}
\partial_t (J\a^T B)=0
\end{equation}
 and hence that
\begin{equation}
\partial_t (J\diva B)= \partial_t \diverge(J\a^T B)=0.
\end{equation}
Here we have used the well-known geometric identity $\partial_j(J\a_{ij})=0$.
Finally,
\begin{equation}
\dt \eta =0\quad\text{on } \p\Omega.
\end{equation}
Since we are interested in showing that $(\eta,\rho,u,B)(t)\rightarrow (Id, \bar \rho,0,\bar B)$ as $t\rightarrow\infty$ in a strong sense. Due to these conservations, we may conclude that
\begin{equation}
\rho J =\bar\rho,\ J\a^T B= \bar B \text{ in }\Omega, \text{ and }\eta=Id \text{ on }\p\Omega.
\end{equation}
Note then that $\diva B =J^{-1}\diverge(J\a^T B)=J^{-1}\diverge\bar B=0$. In turn, to have these we need to assume that the initial data satisfy these conditions; such conditions are necessary for our global well-posedness. We may shift $\eta\rightarrow Id+\eta$, and hence $J={\rm det} (I+\nabla\eta)$, $\a=((I+\nabla\eta)^{-1})^T$, and we rerecord these conserved quantities in the following form:
\begin{equation}\label{imp1}
\rho = \bar\rho J^{-1},\ B = J^{-1}(I+\nabla \eta)\bar B=\bar b J^{-1}(e_3+\p_3 \eta) \text{ in }\Omega, \text{ and }\eta=0 \text{ on }\p\Omega.
\end{equation}
We may refer to \cite{H} for the derivation of the first two identities in \eqref{imp1} in Eulerian coordinates.

Now for the incompressible MHD system \eqref{lagrangianic}, the identity \eqref{jequ} together with the incompressible condition implies
\begin{equation}
\partial_t J =0.
\end{equation}
Again, since we are interested in showing that $(\eta,u,B)(t)\rightarrow (0,0,\bar B)$ as $t\rightarrow\infty$ in a strong sense, we conclude that
\begin{equation}\label{imp2}
J={\rm det} (I+\nabla\eta)=1,\ B =  (I+\nabla \eta)\bar B= \bar b(e_3+\p_3 \eta)  \text{ in }\Omega, \text{ and }\eta=0 \text{ on }\p\Omega.
\end{equation}
We may also refer to \cite{W,LXZ,XZ,AZ} for the first two identities in \eqref{imp2} and \cite{HL} for the derivation in Eulerian coordinates.

\subsection{Reformulation}

The conservation analysis in the previous subsection reveals that in order to have our global well-posedness, the density $\rho$ and the magnetic field $B$ of the compressible fluid and the magnetic field $B$ of the incompressible fluid should have certain relations with the flow map $\eta$. In turn, this motivates us to eliminate them from the systems and then reformulate the systems by using the flow map $\eta$.

We start with the reformulation of the compressible system \eqref{lagrangianc}. We first rewrite the Lorentz force term. Indeed,
\begin{align} 
& {B} \cdot \nabla_\a  {B}-\nabla_\a \left(\frac{| {B}|^2}{2} \right) = {B} \cdot \nabla_\a  {B}-\nabla_\a B_j B_j\nonumber
\\ &\quad = \bar B \cdot\nabla  (B-\bar B)+  \bar B \cdot(\nabla_\a-\nabla) (B-\bar B)+ (B -\bar B ) \cdot\nabla_\a (B-\bar B)\nonumber
\\ &\qquad-\nabla  (B_j-\bar B_j) \bar B_j-(\nabla_\a-\nabla)  (B_j-\bar B_j) \bar B_j -\nabla_\a (B_j-\bar B_j) (B_j-\bar B_j). 
\end{align}
By the second identity in \eqref{imp1}, we obtain
\begin{align} 
&\bar B \cdot\nabla  (B-\bar B)-\nabla  (B_j-\bar B_j) \bar B_j=\bar b\p_3(B-\bar B)-\bar b \nabla(B_3-\bar b)\nonumber
\\&\quad=\bar b^2\left(\p_3( (J^{-1}-1)  e_3+J^{-1}\p_3\eta)-\nabla (J^{-1}-1+J^{-1}\p_3\eta_3 )\right). 
\end{align}
Note that $J^{-1}=1-\diverge \eta+O(|\nabla\eta|^2)$. Then we conclude that
\begin{equation}\label{reb}
  {B} \cdot \nabla_\a  {B}-\nabla_\a \left(\frac{| {B}|^2}{2} \right)
 =  \bar b^2\left( \p_3^2\eta- \p_3\diverge \eta e_3+\nabla\diverge \eta- \nabla\p_3\eta_3\right)+ \mathcal{R}_B^\eta,
\end{equation}
where $\mathcal{R}_B^\eta=O(\nabla\eta\nabla^2\eta)$ is the remainder.
We can also rewrite the pressure term by using the first identity in \eqref{imp1}; indeed, by the Taylor expansion,
\begin{equation}\label{R1}
P(\rho)= P(\bar\rho J^{-1})=P (\bar{\rho} )+P '(\bar{\rho} )\bar\rho (J^{-1}-1)+\int_{\bar{\rho} }^{\bar\rho J^{-1}}(\bar\rho J^{-1}-z)  P ^{\prime\prime}(z)\,dz,
\end{equation}
which implies that
\begin{equation}\label{rebp}
\nabla_\a   P(\rho) =-P '(\bar{\rho} )\bar\rho \nabla \diverge \eta+\mathcal{R}_P^\eta,
\end{equation}
where $\mathcal{R}_P^\eta=O(\nabla\eta\nabla^2\eta)$ is the remainder. By \eqref{reb} and \eqref{rebp}, we can rewrite the momentum equation in  \eqref{lagrangianc} as a parabolic system with a force term induced by the flow map $\eta$, and the system \eqref{lagrangianc} reduces to the following:
\begin{equation}\label{reformulationc}
\begin{cases}
\partial_t\eta=u & \text{in }
\Omega
 \\ \bar \rho J^{-1}\partial_t u -\mu\Delta_\a  {u}-(\mu+\mu')\nabla_\a \diverge_\a u- P '(\bar{\rho} )\bar\rho \nabla \diverge \eta
 &\\\qquad=\kappa\bar b^2\left( \p_3^2\eta- \p_3\diverge \eta e_3+\nabla\diverge \eta- \nabla\p_3\eta_3\right)+ \mathcal{R}^\eta & \text{in }
\Omega \\
u= 0 &\text{on }\p\Omega
\\ (\eta,  u)\mid_{t=0}=( \eta_0, u_0),
\end{cases}
\end{equation}
where the remainder $\mathcal{R}^\eta=\kappa\mathcal{R}_B^\eta-\mathcal{R}_P^\eta=O(\nabla\eta\nabla^2\eta)$.

We now reformulate the incompressible system \eqref{lagrangianic}. By the second identity in \eqref{imp2}, we obtain
\begin{equation}
B\cdot\nabla_\a B=B_j\a_{jk}\partial_kB=\bar{B}_k\partial_k(\bar{B}_m(\delta_{m3}+\partial_m\eta))
 =\bar b^2\p_{3}^2\eta.
\end{equation}
Then the system \eqref{lagrangianic} can be reformulated as a Navier-Stokes system with a force term induced by the flow map $\eta$:
\begin{equation}\label{reformulationic}
\begin{cases}
\partial_t\eta=u & \text{in }
\Omega
 \\\rho_0\partial_t u - \mu\da u+\naba p= \kappa\bar b^2\p_{3}^2\eta& \text{in }
\Omega
\\ \diva u=0 & \text{in }
\Omega \\
u= 0 &\text{on }\p\Omega
\\ (\eta,  u)\mid_{t=0}=( \eta_0, u_0).
\end{cases}
\end{equation}
\begin{remark}
In this paper, we will prove the global well-posedness of the reformulated systems \eqref{reformulationc} and \eqref{reformulationic} around the steady state $(\eta,u)=(0,0)$. With the solution $(\eta,u)$ of \eqref{reformulationc}, defining the density $\rho=\bar\rho\, {\rm det} (I+\nabla\eta) ^{-1}$ and the magnetic field $B = \bar b\,{\rm det} (I+\nabla\eta) ^{-1}(e_3+\p_3 \eta)$, then $(\eta,\rho,u,B)$ solves \eqref{lagrangianic} by imposing the initial conditions that $\rho_0=\bar\rho\, {\rm det} (I+\nabla\eta_0) ^{-1}$ and $B_0= \bar b\,{\rm det} (I+\nabla\eta_0) ^{-1}(e_3+\p_3 \eta_0)$. Similar conclusion holds for the incompressible case.
\end{remark}

\section{Main results}

We take $L^p(\Omega),p\ge 1$ and $H^k(\Omega),k\ge 0$ for the usual $L^p$ and Sobolev spaces on $\Omega$ with norms $\norm{\cdot}_{L^p}$ and $\norm{\cdot}_k$, respectively. We will typically write $H^0 = L^2$. We also introduce the following anisotropic Sobolev norm:
\begin{equation}
\norm{f}_{k,l}:=\sum_{ \al_1+\al_2 \le l}\norm{\p_1^{\al_1}\p_2^{\al_2} f}_k.
\end{equation}

We first state our global well-posedness result for the compressible MHD system \eqref{reformulationc}. For this, we define some energy functionals. For a generic integer $n\ge 3$, we define the energy as
\begin{equation}\label{p_energy_defc}
 \se{n} := \sum_{j=0}^{n}  \ns{\dt^j u}_{2n-2j} +\ns{ \eta}_{2n+1}
\end{equation}
and the dissipation as
\begin{equation}\label{p_dissipation_defc}
 \sd{n} :=   \sum_{j=0}^{n}  \ns{\dt^j u}_{2n-2j+1}
+\ns{\diverge  \eta}_{2n }+\ns{\p_3  \eta}_{2n }+\ns{ \eta}_{2n }.
\end{equation}
We will consider both $n=2N$ and $n=N+2$ for the integer $N\ge 4$. Finally, we define
\begin{equation}\label{G_defc}
\mathcal{G}_{2N}(t) :=\sup_{0 \le r \le t} \se{2N} (r) + \int_0^t \sd{2N} (r) dr + \sup_{0 \le r \le t} (1+r)^{2N-4} \se{N+2} (r).
\end{equation}
Our global well-posedness result of \eqref{reformulationc} is stated as follows.
\begin{theorem}\label{thc}
Let $N\ge 4$ be an integer. Assume that $u_0\in H^{4N}(\Omega)$ and $\eta_0\in H^{4N+1}(\Omega)$ satisfy the appropriate compatibility conditions for the local well-posedness of \eqref{reformulationc} and
\begin{equation}
\eta_0=0\quad\text{on }\p\Omega.
\end{equation}
There exists a constant $\varepsilon_0>0$ such that if $\se{2N} (0) \le \varepsilon_0$, then there exists a global unique solution $(\eta,u)$ solving  \eqref{reformulationc} on $[0,\infty)$. The solution obeys the estimate
\begin{equation}
\mathcal{G}_{2N}(\infty) \ls \se{2N} (0).
\end{equation}
\end{theorem}

We now state our global well-posedness result for the incompressible MHD system \eqref{reformulationic}. For this, we define some energy functionals. For a generic integer $n\ge 3$, we define the energy as
\begin{equation}\label{p_energy_def}
 \se{n} := \sum_{j=0}^{n}  \ns{\dt^j u}_{2n-2j} + \sum_{j=0}^{n-1}  \ns{\nabla \dt^j p}_{2n-2j-2}
+\ns{ \eta}_{1,2n }+\ns{ \eta}_{2n }
\end{equation}
and the dissipation as
\begin{equation}\label{p_dissipation_def}
\begin{split}
 \sd{n} :=& \ns{ u}_{1,2n }+ \ns{ u}_{2n }+ \sum_{j=1}^{n}  \ns{\dt^j u}_{2n-2j+1}
+\ns{ \nabla p}_{2n-2 }+ \sum_{j=1}^{n-1}  \ns{\nabla \dt^j p}_{2n-2j-1}
\\&+\ns{\p_3  \eta}_{0,2n }+\ns{ \eta}_{2n }.
\end{split}
\end{equation}
We will consider both $n=2N$ and $n=N+2$ for the integer $N\ge 4$.  We also define
\begin{equation}\label{fff}
\f:=  \ns{\eta}_{4N+1}\text{ and }\mathcal{J}_{2N}:=\ns{u}_{4N+1}+\ns{ \nabla p}_{4N-1}.
\end{equation}
Finally, we define
\begin{align}\label{G_def} 
\mathcal{G}_{2N}(t) :=& \sup_{0 \le r \le t} \se{2N} (r) + \int_0^t \sd{2N} (r) dr + \sup_{0 \le r \le t} (1+r)^{2N-4} \se{N+2} (r)\nonumber
\\&+ \sup_{0 \le r \le t} \f(r)  + \int_0^t \frac{  \mathcal{J}_{2N}(r)}{(1+r)^{1+\vartheta}}dr 
\end{align}
for any fixed $0<\vartheta\le N-3$ (this requires that $N\ge 4$). Our global well-posedness result of \eqref{reformulationic} is stated as follows.
\begin{theorem}\label{thic}
Let $N\ge 4$ be an integer. Let the parameter density function $\rho_0$ be so that $\nabla\rho_0\in H^{4N-1}(\Omega)$ and $0<\underline{\rho}\le \rho_0\le \bar{\rho}<\infty$ for two constants $\underline{\rho}, \bar{\rho}$. Assume that $u_0\in H^{4N}(\Omega)$ and $\eta_0\in H^{4N+1}(\Omega)$ satisfy the appropriate compatibility conditions for the local well-posedness of \eqref{reformulationic} and
\begin{equation}
 {\rm det} (I+\nabla\eta_0)=1\quad \text{in }\Omega,\text{ and }
\eta_0=0\quad\text{on }\p\Omega.
\end{equation}
There exists a constant $\varepsilon_0>0$ such that if $\se{2N} (0) + \f(0) \le \varepsilon_0$, then there exists a global unique solution $(\eta,u,p)$ solving  \eqref{reformulationic} on $[0,\infty)$. The solution obeys the estimate
\begin{equation}
\mathcal{G}_{2N}(\infty) \ls \se{2N} (0) + \f(0).
\end{equation}
\end{theorem}

\begin{remark}
Note that in Theorem \ref{thic} there is no any smallness assumption for the initial density; this is same as the inhomogeneous incompressible Navier-Stokes equations, see Lady${\rm\check{z}}$henskaya and Solonnikov \cite{LS} for instance.
Note that Theorem \ref{thic} holds also for the homogeneous case $\rho_0\equiv 1$. In both theorems, the bound of $\mathcal{G}_{2N}(\infty)$ implies that $\se{N+2} (t) \ls (1+t)^{-2N+4} $. Since $N$ may be taken to be arbitrarily large, this decay result can be regarded as an ``almost exponential" decay rate. Moreover, $\eta$ is sufficiently small to guarantee that it is a diffeomorphism for each $t\ge 0$. As such, we may change coordinates to $y\in\Omega$ to produce global-in-time, decaying solutions to the original compressible and incompressible MHD systems.
\end{remark}
\begin{remark}
We remark that our global well-posedness heavily relies on that the vertical component of the steady magnetic field is not vanishing; after this paper is completed, we learn that the global well-posedness of the 2D homogeneous incompressible system \eqref{ns_euler0ich} around a horizontal magnetic field is established in \cite{RXZ} very recently. Hence, it would be very interesting to study the viscous non-resistive MHD systems for the remaining cases; furthermore, the final goal of the project would be to study the initial-boundary value problem in a domain with curved boundary, e.g. a smooth bounded domain, and this is still remarkably challenging.
\end{remark}

\begin{remark}
We remark that it is well-known that for the Navier-Stokes equations in Lagrangian coordinates, one can not include the norm of $\eta$ without time derivatives in either the energy or the dissipation and hence one can not expect that $\eta$ is dissipated or decays in time. Such difference between our MHD systems and the Navier-Stokes equations is due to the presence of the second order terms of $\eta$, induced by the magnetic field, in our setting. 
\end{remark}

\begin{remark} 
It is known that the viscous non-resistive MHD systems resemble the viscoelastic systems \cite{Lin,LLZ,LZ1}; roughly speaking, in Lagrangian coordinates, the term $\pa_3^2\eta$ is replaced with $\Delta\eta$. Hence, due to this stronger dissipation, from the proof of our theorems, one may deduce the global well-posedness and exponential decay of the compressible and incompressible viscoelastic systems; moreover, the regularity index $4N$ may  be relaxed to $2$. One may refer to \cite{HX} for a related study of the incompressible homogeneous viscoelastic system.
\end{remark}
\begin{remark}
In a forthcoming paper, we expect to use the two-tier energy method developed in this paper to show the sharp nonlinear stability of the viscous non-resistive MHD Rayleigh-Taylor problem (the linear analysis for the homogeneous incompressible fluids was developed in \cite{W}.); this is our primary motivation to consider the viscous non-resistive MHD systems.
\end{remark}

Recall that the local well-posedness of the original systems \eqref{ns_euler0c} and \eqref{ns_euler0ic} in Sobolev spaces $H^m$ with $m$ sufficiently large is standard \cite{K}; we refer to \cite{F,CMRR} for the local well-posedness of the homogeneous incompressible system \eqref{ns_euler0ich} with the low regularity in $\mathbb{R}^n$ with $n=2,3$. This may produce the local well-posedness of the reformulated systems \eqref{reformulationc} and \eqref{reformulationic}. One can also directly construct the local smooth solutions to the reformulated systems \eqref{reformulationc} and \eqref{reformulationic}; the local well-posedness of \eqref{reformulationc} is again standard, while for \eqref{reformulationic} we can refer to \cite{GT_lwp} for the construction of local solutions. Therefore, by a continuity argument, to prove Theorems \ref{thc} and \ref{thic} it suffices to derive the a priori estimates, namely, Theorems \ref{Apc} and \ref{Ap}, respectively.

We begin with the explanation of the proof of Theorem \ref{Apc} for \eqref{reformulationc}. The basic strategy in the energy method is to use first the basic energy-dissipation structure of the system to get the estimates of $(\eta,u)$ as well its temporal and horizontal spatial derivatives that preserve the boundary conditions. The next step is then to use the structure of the equations to improve the estimates. First, we separate the third component and the first two components of the momentum equation to discover the ODE structure: $\dt f+f=g$ for $f=\pa_3 \diverge\eta$ and $\pa_3^2 \eta_\ast$, respetively; exploring this ODE energy-dissipation structure and interwinding between vertical derivatives and horizontal derivatives, we can improve the dissipation estimates of $(\eta,u)$ without time derivatives and the energy estimates of $\eta$. Then we will employ the elliptic regularity theory of the Lam$\acute{e}$ system for $u$ to improve the energy and dissipations estimates of $ u$. The conclusion is
\begin{equation}
\dtt \se{n}+\sd{n}\le \mathcal{N}_n,
\end{equation}
where $\mathcal{N}_n$ represents the nonlinear estimates. The remaining is to control the nonlinear estimates, and the basic goal is $\mathcal{N}_n\ls \sqrt{\se{n}}\sd{n}$; this would then close the estimates in a small-energy regime. Unfortunately, there is one of nonlinear estimates can not be controlled in this way:
\begin{equation}\label{trouble}
\int_\Omega   \p^\al \eta \cdot  \p^\al G \text{ when }\alpha\in \mathbb{N}^2 \text{ and }|\alpha|=2n,
\end{equation}
where $G$ is the nonlinear term defined by \eqref{317317}.
Indeed, the difficulty lies in that we can only control $\norm{ \eta }_{2n+1}^2$ by $\se{n}$ rather than $\sd{n}$; this would be harmful for the energy method. Our solution to this problem is to implement the two-tier energy method.  The idea is to employ two tiers of energies and dissipations, $\se{N+2}$, $\sd{N+2}$, $\se{2N}$, and $\sd{2N}$.  We then control the troubling terms in \eqref{trouble} by $\sqrt{\se{N+2}}\se{2N}$ when $n=2N$ and by $\sqrt{\se{2N}}\sd{N+2}$ when $n=N+2$. This leads to
\begin{equation}\label{conclu}
 \frac{d}{dt} \se{2N} + \sd{2N} \ls  \sqrt{\se{N+2}}\se{2N} \text{ and } \frac{d}{dt} \se{N+2} + \sd{N+2} \le 0.
\end{equation}
If $\se{N+2}$ decays at a sufficiently fast rate, then the estimates close. This can be achieved by using the second inequality in \eqref{conclu};
although we do not have that $  \se{N+2}\ls \sd{N+2}$, which rules out the exponential decay, we can use an interpolation argument as \cite{RG,GT_per} to bound $\se{N+2} \ls (\se{2N})^{1-\theta} (\sd{N+2})^{\theta}$ for $\theta=(2N-4)/(2N-3)$.  Plugging this in leads to an algebraic decay estimate for $\se{N+2}$ with the rate $(1+t)^{-2N+4}$. Consequently, this scheme of the a priori estimates of Theorem \ref{Apc} closes by setting $N\ge 4$. Full details of the proof will be carried out in Section \ref{sec com}.

We now explain the strategy of the proof of Theorem \ref{Ap} for \eqref{reformulationic}. As already noticed from the definition of the energy functionals, the situation is a bit more complicated than the compressible case. Again, the first step is to use the basic energy-dissipation structure of the system to get the estimates of $(\eta,u)$ as well its temporal and horizontal spatial derivatives. Note that it is essential to employ the structure of the nonlinear terms of $\diverge u$ and $\diverge \eta$ since we can not get any estimates of the pressure $p$ without spatial derivatives. The next step is to use the structure of the equations to improve the estimates. Note that now the pressure $p$ is a new unknown and unlike the Cauchy problem it is not a quadratic term, and so we do not have the ODE structure as the compressible case. Hence, the only way to improve the estimates is to use the elliptic regularity theory of the Stokes system. However, we can not use the Stokes system for $(u,p)$ since we have not controlled $\pa_3^2\eta$ yet. The crucial observation is that we have certain control of the horizontal derivatives of $\eta$; if we write $\pa_3^2\eta=\Delta\eta-\Delta_\ast\eta$ and consider the Stokes system for $(w,p)$ with the quantity $w=  u+  \kappa\bar b^2/\mu \eta$, then we can deduce the desired dissipation estimates of $(\eta,u,p)$ without time derivatives and the energy estimates of $\eta$. Then we will employ the Stokes system for $(u,p)$ to improve the energy and dissipations estimates of $(u,p)$. The conclusion is
\begin{equation}
\dtt \se{n}+\sd{n}\le \mathcal{N}_n,
\end{equation}
and again the difficulty is that $\mathcal{N}_n\ls \sqrt{\se{n}}\sd{n}$ does not hold. We again implement the two-tier energy method to conclude that
\begin{equation}\label{conclu2}
 \frac{d}{dt} \se{2N} + \sd{2N} \ls  \sqrt{\se{N+2}}(\f+\mathcal{J}_{2N}) \text{ and } \frac{d}{dt} \se{N+2} + \sd{N+2} \le 0.
\end{equation}
The second inequality yields the decay estimate for $\se{N+2}$ with the rate $(1+t)^{-2N+4}$.
The control of $\f $ and $\mathcal{J}_{2N} $ is through the following:
\begin{equation}\label{ine}
 \dtt {\mathcal{F}}_{2N}
+  \f +\mathcal{J}_{2N}  \ls    \se{2N}+ \sd{2N}.
\end{equation}
A time weighted analysis on \eqref{ine} leads to the boundedness of $\f$ and $
  \int_0^t \frac{\mathcal{J}_{2N}}{(1+r)^{1+\vartheta}}dr $ for any $\vartheta>0.$
This together with the decay of $\se{N+2}$ then closes the a priori estimates of Theorem \ref{Ap} by choosing $0<\vartheta\le N-3$ for $N\ge 4$.
Full details of the proof will be carried out in Section \ref{sec incom}.
\smallskip\smallskip

\hspace{-12pt}{\bf Notation.}
We now set the conventions for our notation. The Einstein convention of summing over repeated indices is used. Throughout the paper $C>0$ will denote a generic constant that does not depend on the initial data and time, but can depend on $N$, $\Omega$, the steady states, or any of the parameters of the problem.  We refer to such constants as ``universal.''  They are allowed to change from line to line. We employ the notation $A \ls B$ to mean that $A \le C B$ for a universal constant $C>0$, and we write $\dt A+B\ls D$ for $\dt A+CB\le D$. We will write $\mathbb{N} = \{ 0,1,2,\dotsc\}$ for the collection of non-negative integers.  When using space-time differential multi-indices, we will write $\mathbb{N}^{1+m} = \{ \alpha = (\alpha_0,\alpha_1,\dotsc,\alpha_m) \}$ to emphasize that the $0-$index term is related to temporal derivatives.  For just spatial derivatives we write $\mathbb{N}^m$.  For $\alpha \in \mathbb{N}^{1+m}$ we write $\partial^\alpha = \dt^{\alpha_0} \p_1^{\alpha_1}\cdots \p_m^{\alpha_m}.$ We define the parabolic counting of such multi-indices by writing $\abs{\alpha} = 2 \alpha_0 + \alpha_1 + \cdots + \alpha_m.$  We will write $\nabla_\ast$ for the horizontal gradient, $\diverge_\ast$ for the horizontal divergence and $\Delta_\ast$ for the horizontal Laplace operator. For vector $v=(v_1,v_2,v_3)$, we write $v_\ast=(v_1,v_2)$ for the horizontal components. Finally, for a \textit{given norm} $\norm{\cdot}$ and an integer $k\ge 0$, we introduce the following notation for sums of derivatives:
\begin{equation} \nonumber
\norm{\bar{\nab}^k_0 f}^2 := \sum_{\substack{\alpha \in \mathbb{N}^{1+3}, \abs{\alpha}\le k} } \norm{\pa^\al  f}^2.
\end{equation}

\section{Compressible MHD system}\label{sec com}

In this section, we will derive the a priori energy estimates for the smooth solution $(\eta,u)$ to the compressible MHD system \eqref{reformulationc}. We assume throughout the section that the solution obeys the estimate $\mathcal{G}_{2N}(T) \le \delta$ for sufficiently small $\delta>0$.

\subsection{Energy evolution}\label{sec_horizc}

In this subsection we derive energy evolution estimates for temporal and horizontal spatial derivatives by using the energy-dissipation structure of the system.

\subsubsection{Energy evolution of time derivatives}\label{stable2}

For the temporal derivatives, it is a bit more convenient to use the following geometric formulation. Applying $\dt^j$ for $j=0,\dots,n$ to the system \eqref{reformulationc}, we find that
\begin{equation}\label{linear_geometricc}
\begin{cases}
  \dt (\dt^j \eta) =\dt^j u  & \text{in } \Omega
\\  \bar \rho J^{-1}\dt (\dt^j u) -\mu\Delta_\a  (\dt^j u)-(\mu+\mu')\nabla_\a \diverge_\a  (\dt^j u)-P '(\bar{\rho} )\bar\rho \nabla \diverge (\dt^j \eta)
 &\\\quad=\kappa\bar b^2\left( \p_3^2(\dt^j \eta)- \p_3\diverge (\dt^j \eta) e_3+\nabla\diverge (\dt^j \eta)- \nabla\p_3(\dt^j \eta_3)\right)+ F^{j}  & \text{in }
\Omega \\
\dt^j u  =0 & \text{on } \p\Omega,
\end{cases}
\end{equation}
where for $i=1,2,3,$
\begin{align}\label{F_def_startc} 
 F_i^{j}  = &\dt^j(\mathcal{R}_i^\eta)+\sum_{0 < \ell \le j}  C_j^\ell\left\{ (\mu+\mu')
\mathcal{A}_{ik} \p_k ( \dt^\ell  \mathcal{A}_{lm} \dt^{j-\ell}\p_m u_l)
   +   (\mu+\mu')\dt^\ell \mathcal{A}_{ik}\dt^{j-\ell}\p_k ( \mathcal{A}_{lm}\p_m u_l) \right.\nonumber
   \\&  \left.+  \mu\mathcal{A}_{lk} \p_k ( \dt^\ell  \mathcal{A}_{lm} \dt^{j-\ell}\p_m u_i)
   +   \mu\dt^\ell \mathcal{A}_{lk}\dt^{j-\ell}\p_k ( \mathcal{A}_{lm}\p_m u_i) -\bar \rho \dt^\ell(J^{-1})\dt (\dt^{j-\ell} u_i)\right\}. 
\end{align}

We record the estimates of the nonlinear terms $F^j$ in the following lemma.
\begin{lem}\label{p_F_estimatesc}
For $n\ge 3$, it holds that
\begin{equation}\label{p_F_e_001c}
\ns{ F^j}_{0}   \ls  \se{n} \sd{n}.
\end{equation}
\end{lem}
\begin{proof}
Note that all terms in the definitions of $F^{j}$ are at least quadratic; each term can be written in the form $X Y$, where $X$ involves fewer derivative counts than $Y$. We may use the usual Sobolev embeddings along with the definitions of $\se{n}$ and $\sd{n} $ to estimate $\norm{X}_{L^\infty}^2\ls  \se{n} $ and $\norm{Y}_{0}^2\ls   \sd{n} $.
Then $\norm{XY}_0^2\le \norm{X}_{L^\infty}^2\norm{Y}_{0}^2\ls   \se{n} \sd{n} $, and the estimate \eqref{p_F_e_001c} follows.
\end{proof}

For a generic integer $n\ge 3$, we define the temporal energy by
\begin{equation}
 \bar{\mathcal{E}}^{t}_{n} =
 \sum_{j=0}^{n} \left(\bar\rho \norm{\dt^j u}_0^2+P'(\bar\rho)\bar\rho \norm{\dt^j \diverge \eta}_0^2 +\kappa\bar b^2 \norm{\dt^j\p_3\eta_\ast}_0^2
  +\kappa\bar b^2 \norm{\dt^j\diverge_\ast\eta_\ast}_0^2
 \right)
\end{equation}
and the temporal dissipation by
\begin{equation}
 \bar{\mathcal{D}}_n^{t} = \sum_{j=0}^{ n}  \norm{\dt^j u}_1^2.
\end{equation}
Then we have the following energy evolution.
\begin{prop}\label{i_temporal_evolution  Nc}
For $n\ge 3$, it holds that
\begin{equation} \label{i_te_0c}
  \frac{d}{dt} \bar{\mathcal{E}}_{n}^{t}
+ \bar{\mathcal{D}}_{n}^{t}
  \ls   \sqrt{\se{n} } \sd{ n}.
\end{equation}
\end{prop}

\begin{proof}
Taking the dot product of the second equation of $\eqref{linear_geometricc}$ with $J\dt^j u$, $j=0,\dots,n,$ and then integrating by parts, using the boundary condition, we obtain
\begin{align}\label{i_ge_ev_0c} \nonumber
&\hal  \frac{d}{dt} \int_\Omega \bar \rho \abs{\dt^j u}^2
+ \int_\Omega  \mu J\abs{ \naba\dt^j u}^2 +(\mu+\mu') J\abs{ \diva\dt^j u}^2+\int_\Omega P '(\bar{\rho} )\bar\rho  \diverge (\dt^j \eta) \diverge(\dt^j u)\nonumber
\\&\quad+\int_\Omega\kappa\bar b^2\left(
  \p_3 (\dt^j \eta)\cdot \p_3(\dt^j u)-\diverge (\dt^j \eta) \p_3(\dt^j u_3)+ \diverge (\dt^j \eta)\diverge(\dt^j u)- \p_3(\dt^j \eta_3)\diverge (\dt^j u)\right)\nonumber
\\&\quad =\int_\Omega J\dt^j u\cdot F^{j}-\int_\Omega P '(\bar{\rho} )\bar\rho  \diverge (\dt^j \eta) \diverge((J-1)\dt^j u)\nonumber
\\&\qquad- \int_\Omega
\kappa\bar b^2\left( \p_3 (\dt^j \eta)\cdot \p_3((J-1)\dt^j u)-\diverge (\dt^j \eta) \p_3((J-1)\dt^j u_3) \right.\nonumber
\\&\qquad\qquad\qquad \left.  +\diverge (\dt^j \eta)\diverge((J-1)\dt^j u)- \p_3(\dt^j \eta_3)\diverge ((J-1)\dt^j u)\right). 
\end{align}
By the first equation, we get
\begin{equation}\label{i_te_2c}
\int_\Omega   \diverge (\dt^j \eta) \diverge(\dt^j u) = \int_\Omega   \diverge (\dt^j \eta) \diverge(\dt^{j+1} \eta)
=\hal\dtt \int_\Omega   \abs{ \diverge (\dt^j \eta) }^2
\end{equation}
and
\begin{align}\label{i_te_2c1} 
&\int_\Omega
 \left( \p_3 (\dt^j \eta)\cdot \p_3(\dt^j u)-\diverge (\dt^j \eta) \p_3(\dt^j u_3)+ \diverge (\dt^j \eta)\diverge(\dt^j u)- \p_3(\dt^j \eta_3)\diverge (\dt^j u)\right)\nonumber
\\&\quad=\hal\dtt \int_\Omega
 \left( \abs{\dt^j \p_3\eta}^2 -2\diverge (\dt^j \eta) \p_3(\dt^j \eta_3)+ \abs{\diverge (\dt^j \eta)}^2\right)\nonumber
\\&\quad=\hal\dtt \int_\Omega
 \left( \abs{\dt^j \p_3\eta_\ast}^2  + \abs{\dt^j \diverge_\ast  \eta_\ast}^2\right). 
\end{align}

We now estimate the right hand side of \eqref{i_ge_ev_0c}. For the $F^{j}$ term, by \eqref{p_F_e_001c}, we may bound
\begin{equation}
\int_\Omega  J\dt^j u\cdot F^{j} \ls   \norm{\dt^j u}_{0}   \norm{F^{j}}_0 \ls  \sqrt{\sd{n} } \sqrt{\se{n} \sd{n}}
 .
\end{equation}
The remaining terms can be bounded by
\begin{equation}\label{i_te_5c}
 \norm{\dt^j \eta}_{1}  \norm{ \eta}_{3} \norm{ \dt^ju }_{1}\ls   \sqrt{\sd{n} } \sqrt{\se{n} } \sqrt{\sd{n} }  .
\end{equation}

Now we combine \eqref{i_te_2c}--\eqref{i_te_5c} to deduce from \eqref{i_ge_ev_0c} that, summing over $j$,
\begin{equation} \label{iiiic}
 \hal \frac{d}{dt}\bar{\mathcal{E}}_{n}^{t}
+  \sum_{j=0}^{ n}  \int_\Omega  \mu J\abs{ \naba\dt^j u}^2 +(\mu+\mu') J\abs{ \diva\dt^j u}^2
  \ls   \sqrt{\se{n} } \sd{ n}.
\end{equation}
We then seek to replace $ J\abs{ \nabla_{\mathcal{A}} \dt^{j} u}^2$ with $\abs{\nabla \dt^{j} u}^2$ and $ J\abs{ \diverge_{\mathcal{A}} \dt^{j} u}^2$ with $\abs{\diverge \dt^{j} u}^2$
in \eqref{iiiic}.  To this end we write
\begin{equation}\label{i_te_8c}
 J\abs{ \naba \dt^{j} u}^2 = \abs{\nabla \dt^{j} u}^2+(J-1)\abs{\nabla \dt^{j} u}^2  +   J\left(\naba \dt^{j} u + \nabla\dt^{j} u\right): \left(\naba \dt^{j} u - \nabla \dt^{j} u\right)
\end{equation}
and
\begin{equation}\label{i_te_8c2}
 J\abs{ \diverge_\a \dt^{j} u}^2 = \abs{\diverge \dt^{j} u}^2+(J-1)\abs{\diverge \dt^{j} u}^2  +   J\left(\diva \dt^{j} u + \diverge\dt^{j} u\right) \left(\diva \dt^{j} u - \diverge \dt^{j} u\right).
\end{equation}
Hence,
\begin{equation}\label{i_te_9c}
 \sum_{j=0}^{ n}  \int_\Omega  \mu J\abs{ \naba\dt^j u}^2 +(\mu+\mu') J\abs{ \diva\dt^j u}^2
 \ge \sum_{j=0}^{ n}  \int_\Omega  \mu  \abs{ \nabla\dt^j u}^2 +(\mu+\mu') \abs{ \diverge\dt^j u}^2
 -C  \sqrt{\se{n}} \sd{n}.
\end{equation}
We may then use \eqref{i_te_9c} to replace in \eqref{iiiic} and derive \eqref{i_te_0c}, by \eqref{visco} and using Poincar\'e's inequality.
\end{proof}

\subsubsection{Energy evolution of horizontal derivatives}

For the horizontal spatial derivatives, we shall use the following linear perturbed formulation
\begin{equation}\label{perturbc}
\begin{cases}
\partial_t\eta=u & \text{in }
\Omega
 \\ \bar\rho\partial_t u -\mu\Delta  {u}-(\mu+\mu')\nabla \diverge u -P '(\bar{\rho} )\bar\rho \nabla \diverge \eta
 &\\\quad=\kappa\bar b^2\left( \p_3^2\eta- \p_3\diverge \eta e_3+\nabla\diverge \eta- \nabla\p_3\eta_3\right)+ G & \text{in }
\Omega \\
u= 0 &\text{on }\p\Omega,
\end{cases}
\end{equation}
where for $i=1,2,3,$
\begin{align}\label{317317} 
G_i=&\mathcal{R}_i^\eta+(\mu+\mu') \a_{ij}\p_j\a_{kl}\p_l u_k+(\mu+\mu')(\a_{ij}\a_{kl}-\delta_{ij}\delta_{kl})\p_j\p_l u_k\nonumber
\\&+\mu \a_{jk}\p_k\a_{jl}\p_l u_i+\mu(\a_{jk}\a_{jl}-\delta_{jk}\delta_{jl})\p_k\p_l u_i-\bar \rho (J^{-1}-1)\partial_t u_i. 
\end{align}

We record the estimates of the nonlinear term $G$ in the following lemma.
\begin{lem}\label{p_G_estimatesc}
For $n\ge 3$, it holds that
 \begin{equation}\label{p_G_e_0c}
  \ns{ \bar{\nab}_0^{2n-2} G}_{0}   \ls  (\se{n})^2
 \end{equation}
and
\begin{equation}\label{p_G_e_001c}
\ns{ \bar{\nab}_0^{2n-1} G}_{0}   \ls  \se{n}\sd{n}.
\end{equation}
\end{lem}
\begin{proof}
Note that all terms in the definitions of $G$ are at least quadratic. We apply these
space-time differential operators to $G$ and then expand using the Leibniz rule; each product in
the resulting sum is also at least quadratic. We then write each term in the form $X Y$, where $X$ involves fewer derivative counts than $Y$. Then the estimate \eqref{p_G_e_001c} follows similarly as Lemma \ref{p_F_estimatesc} with a slight modification when $X=\nabla^{2n+1} \eta$; in such cases, we estimate $\norm{\nabla^{2n+1}\eta}_{0}^2\ls \se{n}$ and $\norm{Y}_{L^\infty}^2\ls \sd{n}$. The estimate \eqref{p_G_e_0c} follows more easily.
\end{proof}

For a generic integer $n\ge 3$, we define the  horizontal energy by
\begin{equation}
 \bar{\mathcal{E}}^{\ast}_{n} =  \bar\rho \norm{\nabla_\ast u}_{0,2n-1}^2+P'(\bar\rho)\bar\rho \norm{ \nabla_\ast \diverge \eta}_{0,2n-1}^2 +\kappa\bar b^2 \norm{ \nabla_\ast \p_3\eta_\ast}_{0,2n-1}^2
  +\kappa\bar b^2 \norm{ \nabla_\ast\diverge_\ast\eta_\ast}_{0,2n-1}^2
\end{equation}
and the horizontal dissipation by
\begin{equation}
 \bar{\mathcal{D}}_n^{\ast} = \norm{ \nabla_\ast  u }_{1,2n-1}^2 .
\end{equation}
Then we have the following energy evolution.

\begin{prop}\label{p_upper_evolution  N12c}
For $n\ge 3$, it holds that
\begin{equation}\label{p_u_e_00c}
\frac{d}{dt}\bar{\mathcal{E}}^{\ast}_{n}+\bar{\mathcal{D}}_{n}^{\ast}\ls\sqrt{ \se{n}  } \sd{n} .
\end{equation}
\end{prop}
\begin{proof}
We take $\alpha\in \mathbb{N}^{1+2}$ so that $1\le |\al|\le 2n$. Applying $\p^\al $ to the second equation of \eqref{perturbc} and then taking the dot product with $\p^\al u$, as in Proposition \ref{i_temporal_evolution  Nc}, we find that
\begin{align} \label{p_u_e_1110c} 
 &\hal \frac{d}{dt}\left(  \int_\Omega \bar \rho \abs{\p^\al   u }^2  +P'(\bar\rho)\bar\rho \abs{ \p^\al \diverge  \eta }^2+\kappa\bar b^2 \abs{\p^\al\p_3   \eta_\ast }^2 +\kappa\bar b^2 \abs{ \p^\al \diverge_\ast  \eta_\ast }^2 \right)\nonumber
  \\&\quad+ \int_\Omega  \mu  \abs{ \nabla\p^\al u}^2 +(\mu+\mu') \abs{ \diverge\p^\al u}^2
= \int_\Omega   \p^\al  u  \cdot  \p^\al  G  . 
\end{align}

We now estimate the right side of \eqref{p_u_e_1110c}. Since $\abs{\al}\ge 1$, we may write $\alpha = \gamma +(\alpha-\gamma)$ for some $\gamma \in \mathbb{N}^2$ with $\abs{\gamma}=1$. We can then integrate by parts and use \eqref{p_G_e_001c} to have
\begin{align}\label{i_de_4c} 
  \int_\Omega     \p^\al   u \cdot   \p^\al  G   &=- \int_\Omega     \p^{\alpha+\gamma}   u \cdot   \p^{\alpha-\gamma}  G
 \le \norm{\p^{\alpha+\gamma}   u}_{0}  \norm{ \p^{\alpha-\gamma}   G  }_{0}\nonumber \\&
\le \norm{\p^{\alpha}   u}_{1}  \norm{    G  }_{2n-1 }
\ls \sqrt{ \sd{n} } \sqrt{ \se{n} \sd{n}} . 
\end{align}
The estimate \eqref{p_u_e_00c} then follows from \eqref{p_u_e_1110c} by summing over such $\alpha$ and using Poincar\'e's inequality.
\end{proof}

\subsubsection{Energy evolution controlling $\eta$}

Note that the dissipation estimates in $\bar{\mathcal{D}}_n^t$ of Proposition \ref{i_temporal_evolution  Nc} and $\bar{\mathcal{D}}_n^\ast$ of Proposition \ref{p_upper_evolution  N12c} only contain $u$, in this subsubsection we will recover certain dissipation estimates of $\eta$. Since $\dt \eta =u$, we then only need to estimate for $\eta$ without time derivatives. The estimates result from testing the linear perturbed formulation \eqref{perturbc} by $\eta$. Note that the boundary condition that $\eta=0$ on $\p\Omega$ guarantees this.

We will need a technical lemma to estimate the term $\int_\Omega\p^\al \eta \cdot  \p^\al G$ when $\alpha$ is the highest horizontal derivative.
\begin{lem}\label{p_G_estimatesc22}
It holds that
\begin{equation}\label{pgn1}
\norm{  G }_{4N-1}^2     \ls (\sd{2N})^2+  \se{N+2}    \se{2N}
\end{equation}
and
\begin{equation}\label{pgn2}
\norm{  G }_{2(N+2)}^2   \ls\se{2N}   \sd{N+2}.
\end{equation}
\end{lem}
\begin{proof}
We again write each term of $\p^\alpha G$ in the form $X Y$, where $X$ involves fewer derivative counts than $Y$.
To derive \eqref{pgn1}, we estimate both $\norm{X}_{0}^2\ls \sd{2N}$ and $\norm{Y}_{L^\infty}^2\ls \sd{2N}$ except the cases when $X=\nabla^{2n+1}\eta$; in such cases, we estimate $\norm{\nabla^{2n+1}\eta}_{0}^2\ls \se{2N}$ and $\norm{Y}_{L^\infty}^2\ls \se{N+2}$. Then the estimate \eqref{pgn1} follows.

The estimate \eqref{pgn2}
 follows by estimating $\norm{X}_{L^\infty}^2\ls \se{2N}$ and $\norm{Y}_{0}^2\ls \sd{N+2}$.
 \end{proof}

For a generic integer $n\ge 3$, we define the recovering energy by
\begin{equation}
 \bar{\mathcal{E}}^\sharp_{n} =    \mu  \norm{ \nabla  \eta}_{0,2n}^2 +(\mu+\mu')\norm{ \diverge \eta}_{0,2n}^2
\end{equation}
and the corresponding dissipation by
\begin{equation}
 \bar{\mathcal{D}}_n^\sharp =  \norm{   \diverge \eta}_{0,2n }^2 +  \norm{ \p_3\eta}_{0,2n }^2
  +  \norm{  \eta }_{0,2n}^2.
\end{equation}
Then we have the following energy evolution.

\begin{prop}\label{p_upper_evolution  N'132c}
It holds that
\begin{equation}\label{p_u_e_00'132c}
\frac{d}{dt}\left(\bar{\mathcal{E}}^\sharp_{2N}+2\sum_{\substack{\al\in \mathbb{N}^2\\  |\al|\le 4N}}\int_\Omega \bar \rho  \partial^\alpha    u \cdot \partial^\alpha \eta\right)+\bar{\mathcal{D}}_{2N}^\sharp\ls\sqrt{ \se{2N}   } \sd{2N} +\sqrt{ \se{N+2} }\se{2N}+\bar{\mathcal{D}}_{2N}^t+\bar{\mathcal{D}}_{2N}^{\ast}
\end{equation}
and
\begin{equation}\label{p_u_e_00'132c2}
\frac{d}{dt}\left(\bar{\mathcal{E}}^\sharp_{N+2}+2\sum_{\substack{\al\in \mathbb{N}^2\\  |\al|\le 2(N+2)}}\int_\Omega \bar \rho  \partial^\alpha    u \cdot \partial^\alpha \eta\right)+\bar{\mathcal{D}}_{N+2}^\sharp\ls\sqrt{ \se{2N}   } \sd{N+2}+\bar{\mathcal{D}}_{N+2}^t +\bar{\mathcal{D}}_{N+2}^{\ast}.
\end{equation}
\end{prop}
\begin{proof}
We let $n$ denote either $2N$ or $N+2$ throughout the proof. Applying $\partial^\alpha$ with $\al\in \mathbb{N}^{2}$ so that $ |\al|\le 2n$ to the second equation of \eqref{perturbc} and then taking the dot product with $\p^\al \eta$, since $\eta=0$ on $\p\Omega$,  we find that
\begin{align} \label{p_u_e_111'c} 
 &\int_\Omega \bar \rho \dt (\partial^\alpha    u) \cdot \partial^\alpha \eta
  + \hal\dtt\int_\Omega  \mu  \abs{ \nabla\p^\al \eta}^2 +(\mu+\mu') \abs{ \diverge\p^\al \eta}^2\nonumber
\\&\quad+  \int_\Omega  P'(\bar\rho)\bar\rho \abs{ \p^\al \diverge  \eta }^2+\kappa\bar b^2\abs{\p^\al\p_3   \eta_\ast }^2 +  \kappa\bar b^2\abs{ \p^\al \diverge_\ast  \eta_\ast }^2
 =  \int_\Omega   \p^\al \eta  \cdot  \p^\al G . 
\end{align}

For the first term on the left hand side of \eqref{p_u_e_111'c}, we integrate by parts in time to obtain
\begin{equation}
\int_\Omega \bar \rho \dt (\partial^\alpha    u) \cdot \partial^\alpha \eta
=\dtt\int_\Omega \bar \rho  \partial^\alpha    u \cdot \partial^\alpha \eta- \int_\Omega \bar \rho   \partial^\alpha  u\partial^\alpha  \dt\eta=\dtt\int_\Omega \bar \rho  \partial^\alpha    u \cdot \partial^\alpha \eta- \int_\Omega \bar \rho  \abs{\partial^\alpha u }^2.
\end{equation}
For the right hand side of \eqref{p_u_e_111'c}, we first consider the case $n=2N$. If $\abs{\al}\le 4N-1$, then we use \eqref{p_G_e_001c} to have
\begin{equation}
\int_\Omega   \p^\al \eta \cdot  \p^\al G     \ls  \norm{\p^{\alpha}  \eta}_{0}  \norm{   G }_{4N-1}
 \ls \sqrt{ \sd{2N}}\sqrt{ \se{2N} \sd{2N}}.
\end{equation}
If $\abs{\al}= 4N$, we may write $\alpha = \gamma +(\alpha-\gamma)$ for some $\gamma \in \mathbb{N}^2$ with $\abs{\gamma}=1$. We can then integrate by parts and use \eqref{pgn1} to have
\begin{align} 
  \int_\Omega     \p^\al   \eta \cdot   \p^\al  G   &=- \int_\Omega     \p^{\alpha+\gamma}   \eta \cdot   \p^{\alpha-\gamma}  G
 \le \norm{ \eta}_{4N+1}  \norm{   G  }_{4N-1}\nonumber \\&
\ls \sqrt{ \se{2N} }\left(\sd{2N}+ \sqrt{ \se{N+2} \se{2N}} \right). 
\end{align}
Now for the case $n=N+2$, we use \eqref{pgn2} to estimate
\begin{equation}
\int_\Omega   \p^\al \eta \cdot  \p^\al G
 \ls \norm{\p^{\alpha}  \eta}_{0}  \norm{   G }_{2(N+2)}\ls \sqrt{ \sd{N+2} }\sqrt{\se{2N}\sd{N+2}}.
\end{equation}

Consequently, the estimates \eqref{p_u_e_00'132c} and \eqref{p_u_e_00'132c2} follow by collecting the estimates, summing over such $\alpha$ and using Poincar\'e's inequality.
\end{proof}

\subsection{Improved Estimates}

We now explore the structure of the linear perturbed formulation \eqref{perturbc} to improve the energy-dissipation estimate with the energy evolution in hand.

\subsubsection{ODE regularity}

There are some ODE structures that allow us to get the estimates of vertical derivatives of the solution. Taking the third component of the momentum equation in \eqref{perturbc}, we obtain
\begin{equation}\label{qode}
 (2\mu+\mu')\dt \p_3 q +P '(\bar{\rho} )\bar\rho \p_3 q
 =  - \bar \rho\partial_t u_3 +\mu\Delta_\ast \tilde{u}_3-\mu\p_3\diverge_\ast u_\ast+G_3.
\end{equation}
Here, for simplification of presentation, we have introduced the quantity $q:=-\diverge \eta$ for the ``density perturbation" (and hence $\dt q=-\diverge u$). Such structure was first exploited by \cite{MN83} in the study of the initial boundary value problem for the compressible Navier-Stokes equations. On the other hand, taking the first two components of the momentum equation in \eqref{perturbc}, we have
\begin{equation}\label{etaode}
 -\mu\p_3^2  {u}_\ast-\kappa\bar b^2 \p_3^2  \eta_\ast
= -\bar \rho\partial_t u_\ast +\mu\Delta_\ast u_\ast  -(\mu+\mu')\nabla_\ast \dt q - P '(\bar{\rho} )\bar\rho \nabla_\ast q-\kappa\bar b^2\left( \nabla_\ast q+ \nabla_\ast \p_3\eta_3\right)+G_\ast .
\end{equation}
Note that $u=\dt \eta$, and these two equations resemble the ODE $\dt f + f= g$, up to some errors, and this ODE displays natural energy-dissipation structure.

\begin{prop}\label{qqestprop}
For $n\ge 3$, there exists an energy $\mathfrak{E}_n$ which is equivalent to the sum $ \norm{ \pa_3    \q }_{2n-1}^2+ \norm{  \p_3^2   \eta_\ast }_{2n-1}^2$ such that
 \begin{equation}\label{qqest}
 \dtt\mathfrak{E}_n  +  \norm{   u  }_{2n+1}^2+ \norm{ \diverge \eta }_{2n}^2
+ \norm{  \p_3   \eta  }_{2n}^2+ \norm{     \eta  }_{2n}^2
 \ls { \se{n}  }\sd{n}
+\bar{\mathcal{D}}_n^t+\bar{\mathcal{D}}_n^\ast+\bar{\mathcal{D}}_n^\sharp+\norm{\pa_t  u  }_{2n-1}^2.
\end{equation}
\end{prop}
\begin{proof}
We fix $0\le k\le 2n-1$. We first take the norm $\norm{\cdot}_{k,2n-k-1}^2$ of the equation \eqref{qode} to have
\begin{align}\label{q1} 
&\frac{(2\mu+\mu') P '(\bar{\rho} )\bar\rho}{2}  \dtt \norm{ \pa_3    \q }_{k,2n-k-1}^2  +(P '(\bar{\rho} )\bar\rho)^2 \norm{ \pa_3   \q }_{k,2n-k-1}^2
 +(2\mu+\mu')^2 \norm{ \pa_3  \dt  \q }_{k,2n-k-1}^2\nonumber
\\&\quad =\norm{- \bar \rho\partial_t u_3 +\mu\Delta_\ast \tilde{u}_3-\mu\p_3\diverge_\ast u_\ast+G_3}_{k,2n-k-1}^2\nonumber
\\&\quad\ls   \norm{\pa_t  u_3 }_{2n-1}^2
+  \norm{    u }_{k+1,2n-k}^2+\norm{G_3}_{2n-1}^2 . 
\end{align}
Since
\begin{equation}\label{q11}
 \norm{   q}_{ 0,2n}^2 = \norm{\diverge \eta}_{ 0,2n}^2 \ls \bar{\mathcal{D}}_n^\sharp
\end{equation}
and
\begin{equation}\label{q12}
 \norm{\dt q}_{ 0,2n}^2 = \norm{\diverge u}_{ 0,2n}^2 \le \norm{   u}_{1,2n}^2\ls \bar{\mathcal{D}}_n^t+ \bar{\mathcal{D}}_n^\ast,
\end{equation}
and note that
 \begin{equation}
 \norm{\p_3^2 \eta_3  }_{k,2n-k-1}^2 = \norm{\p_3 (q+ \diverge_\ast \eta_\ast ) }_{k,2n-k-1}^2 \le \norm{ q}_{k+1,2n-k-1}^2 + \norm{\p_3   \eta_\ast }_{k,2n-k}^2
\end{equation}
and
 \begin{equation}
 \norm{\p_3^2 u_3  }_{k,2n-k-1}^2 = \norm{\p_3 (\dt q+ \diverge_\ast u_\ast ) }_{k,2n-k-1}^2 \le\norm{ \dt q}_{k+1,2n-k-1}^2 + \norm{  u_\ast }_{k+1,2n-k}^2,
\end{equation}
we deduce from \eqref{q1} that
\begin{align}\label{q2} 
&  \dtt  \norm{ \pa_3    \q }_{k,2n-k-1}^2  +  \norm{   \q }_{k+1,2n-k-1}^2
 +  \norm{  \dt  \q }_{k+1,2n-k-1}^2+ \norm{\p_3^2 u_3  }_{k,2n-k-1}^2+ \norm{\p_3^2 \eta_3  }_{k,2n-k-1}^2\nonumber
\\&\quad\ls   \norm{\pa_t  u_3 }_{2n-1}^2
+  \norm{    u }_{k+1,2n-k}^2+ \norm{\p_3   \eta_\ast }_{k,2n-k}^2+\norm{G_3}_{2n-1}^2 +\bar{\mathcal{D}}_n^\sharp. 
\end{align}
Next, we take the norm $\norm{\cdot}_{k,2n-k-1}^2$ of \eqref{etaode} to obtain
\begin{align}\label{q3} 
&\frac{\mu\kappa\bar b^2}{2}\frac{d}{dt}   \norm{  \p_3^2   \eta_\ast }_{k,2n-k-1}^2  + \kappa^2\bar b^4 \norm{  \p_3^2   \eta_\ast }_{k,2n-k-1}^2
 + \mu^2\norm{ \p_3^2  u_\ast }_{k,2n-k-1}^2\nonumber
  \\&\quad=\norm{-\bar \rho\partial_t u_\ast +\mu\Delta_\ast u_\ast  -(\mu+\mu')\nabla_\ast \dt q - P '(\bar{\rho} )\bar\rho \nabla_\ast q-\kappa\bar b^2\left( \nabla_\ast q+ \nabla_\ast \p_3\eta_3\right)+G_\ast}_{k,2n-k-1}^2\nonumber
\\
&\quad\ls  \norm{\pa_t  u_\ast  }_{2n-1}^2
+  \norm{   u_\ast }_{k,2n-k+1}^2+\norm{ \dt q }_{k,2n-k}^2+\norm{  q }_{k,2n-k}^2+\norm{   \p_3\eta_3 }_{k,2n-k}^2 +\norm{G_\ast}_{2n-1}^2.
\end{align}

Now combining \eqref{q2} and \eqref{q3} yields that for $k=0,\dots, 2n-1$,
\begin{align} 
& \frac{d}{dt} \left(  \norm{ \pa_3    \q }_{k,2n-k-1}^2+ \norm{  \p_3^2   \eta_\ast }_{k,2n-k-1}^2 \right)\nonumber
\\&\quad+\norm{ \dt q }_{k+1,2n-k-1}^2+\norm{  q }_{k+1,2n-k-1}^2+  \norm{  \p_3   \eta  }_{k+1,2n-k-1}^2
 + \norm{   u  }_{k+2,2n-k-1}^2
\\&\quad\ls  \norm{ \dt q }_{k,2n-k}^2+\norm{  q }_{k,2n-k}^2+  \norm{   u  }_{k+1,2n-k}^2+\norm{   \p_3\eta  }_{k,2n-k}^2 +\norm{\pa_t  u  }_{2n-1}^2
+\norm{G}_{2n-1}^2+ \bar{\mathcal{D}}_n^\sharp.\nonumber
\end{align}
By this recursive inequality on $k$, we conclude that there exist constants $\lambda_k>0,\ k=0,\dots, 2n-1$ such that, by \eqref{q11} and \eqref{q12} again,
 \begin{align}\label{q13} 
& \frac{d}{dt} \sum_{k=0}^{2n-1}\lambda_k\left(   \norm{ \pa_3    \q }_{k,2n-k-1}^2+ \norm{  \p_3^2   \eta_\ast }_{k,2n-k-1}^2 \right)\nonumber
\\&\quad+\sum_{k=0}^{2n-1}\left(\norm{ \dt q }_{k+1,2n-k-1}^2+ \norm{  q }_{k+1,2n-k-1}^2
  +   \norm{  \p_3   \eta  }_{k+1,2n-k-1}^2+  \norm{   u  }_{k+2,2n-k-1}^2\right)\nonumber
\\&\quad\ls  \norm{ \dt q }_{0,2n}^2+\norm{  q }_{0,2n}^2+  \norm{   u  }_{1,2n}^2+\norm{   \p_3\eta  }_{0,2n}^2 +\norm{\pa_t  u  }_{2n-1}^2
+\norm{G}_{2n-1}^2+\bar{\mathcal{D}}_n^\sharp\nonumber
\\&\quad\ls  \norm{\pa_t  u  }_{2n-1}^2
+\norm{G}_{2n-1}^2+\bar{\mathcal{D}}_n^t+\bar{\mathcal{D}}_n^\ast+\bar{\mathcal{D}}_n^\sharp. 
\end{align}
Hence if we define
\begin{equation}
\mathfrak{E}_n:=\sum_{k=0}^{2n-1}\lambda_k\left( \norm{ \pa_3    \q }_{k,2n-k-1}^2+ \norm{  \p_3^2   \eta_\ast }_{k,2n-k-1}^2 \right),
\end{equation}
then $\mathfrak{E}_n$ is equivalent to the sum $\norm{ \pa_3    \q }_{2n-1}^2+ \norm{  \p_3^2   \eta_\ast }_{2n-1}^2 $. \eqref{q13} implies in particular that
 \begin{align} 
&  \dtt\mathfrak{E}_n + \norm{ \dt q }_{2n}^2+ \norm{  q }_{2n}^2
+ \norm{  \p_3   \eta  }_{2n}^2+  \norm{   u  }_{2n+1}^2\nonumber
\\&\quad\ls  \norm{\pa_t  u  }_{2n-1}^2
+\norm{G}_{2n-1}^2+\bar{\mathcal{D}}_n^t+\bar{\mathcal{D}}_n^\ast+\bar{\mathcal{D}}_n^\sharp. 
\end{align}
Using \eqref{p_G_e_001c} to estimate $\norm{G}_{2n-1}^2\ls { \se{n}  }\sd{n}$  we conclude the estimate \eqref{qqest} by recalling that $q=-\diverge\eta$ and using Poincar\'e's inequality.
\end{proof}

\subsubsection{Elliptic regularity}

We now explore the elliptic regularity of the following Lam$\acute{e}$ system:
\begin{equation}\label{stokesp1c}
\begin{cases}
-\mu\Delta  {u}-(\mu+\mu')\nabla\diverge u= -\bar \rho\partial_t  u + P '(\bar{\rho} )\bar\rho \nabla \diverge  \eta
 &\\\quad+\kappa\bar b^2\left( \p_3^2 \eta- \p_3\diverge  \eta e_3+\nabla\diverge  \eta- \nabla\p_3 \eta_3\right) +  G& \text{in }
\Omega
\\   u=0 &\text{on }\p\Omega.
\end{cases}
\end{equation}

Note that the dissipation estimates of the solution without time derivatives have been already controlled in the estimate \eqref{qqest} of Proposition \ref{qqestprop}. We then improve the dissipation estimates of time derivatives.
\begin{prop}\label{p_upper_evolution  N'c}
For $n\ge 3$, it holds that
\begin{equation}\label{d2nc}
\sum_{j=1}^{n}\norm{ \dt^j u  }_{2n-2j+1}^2\ls  { \se{n}  }\sd{n}+\bar{\mathcal{D}}_{n}^t+\norm{ u   }_{2n-1 }^2 .
\end{equation}
\end{prop}
\begin{proof}
We compactly write
\begin{equation}\label{nota1}
 \mathcal{Y}_n =  \ns{ \bar{\nab}_0^{2n-1} G}_{0}.
\end{equation}

Applying the time derivatives $\dt^{j},\ j=1,\dots,n-1$ to the problem \eqref{stokesp1c} and then employing the elliptic estimates \eqref{i_linear_lame} of Lemma \ref{lame eq} with $r=2n-2j+1\ge 3$ to the resulting problems, using the notation \eqref{nota1}, we obtain
\begin{align} \label{fes1}
\norm{\dt^j  u  }_{2n-2j+1}^2
 &\ls
\norm{ \dt^{j+1} u   }_{2n-2j-1 }^2+\norm{ \nabla^2 \dt^{j }\eta}_{2n-2j-1 }^2+ \norm{ \dt^j G   }_{2n-2j-1}^2\nonumber
\\&\ls   \norm{\dt^{j+1} u   }_{2n-2(j+1)+1 }^2 + \norm{\dt^{j }\eta  }_{2n-2 j +1 }^2 +  \mathcal{Y}_n. 
\end{align}
A simple induction on \eqref{fes1} yields, since $\dt \eta=u$
\begin{align}\label{eses11c} 
  \sum_{j=1}^{n}\norm{ \dt^j u  }_{2n-2j+1}^2
 & \ls \ns{\dt^{n} u}_{1}+ \sum_{j=1}^{n-1}\norm{\dt^{j }\eta  }_{2n-2 j +1 }^2   +  \mathcal{Y}_n\nonumber
\\&   \ls   \sum_{j=1}^{n-1}\norm{\dt^{j-1} u   }_{2n-2(j-1)-1 }^2 + \bar{\mathcal{D}}_{n}^t+\mathcal{Y}_n\nonumber
\\&  = \sum_{j=0}^{n-2}\norm{\dt^{j } u   }_{2n-2j-1 }^2   + \bar{\mathcal{D}}_{n}^t+\mathcal{Y}_n. 
\end{align}
Using the Sobolev interpolation and Young's inequality, we can improve \eqref{eses11c} to be
\begin{equation}
  \sum_{j=1}^{n}\norm{ \dt^j u  }_{2n-2j+1}^2
  \ls   \norm{  u   }_{2n-1}^2   +  \sum_{j=1}^{n-2}\norm{\dt^{j } u   }_{0}^2   + \bar{\mathcal{D}}_{n}^t+\mathcal{Y}_n\ls \norm{  u   }_{2n-1}^2   + \bar{\mathcal{D}}_{n}^t+\mathcal{Y}_n.
\end{equation}
Using \eqref{p_G_e_001c} to estimate $\mathcal{Y}_n\ls { \se{n}  }\sd{n}$, we then conclude \eqref{d2nc}.
\end{proof}

Now we improve the energy estimates.
\begin{prop}\label{e2ncprop}
For $n\ge 3$, it holds that
\begin{equation}\label{e2nc}
{\mathcal{E}}_{n}  \lesssim
\bar{\mathcal{E}}_n^t  +\bar{\mathcal{E}}_n^\sharp  +\mathfrak{E}_n + (\mathcal{E}_{n} )^{2}.
\end{equation}
\end{prop}
\begin{proof}
We compactly write
\begin{equation}\label{nota2}
 \mathcal{X}_n =   \ns{ \bar{\nab}^{2n-2}_0  G}_{0} .
\end{equation}

Applying the time derivatives $\dt^{j},\ j=0,\dots,n-1$ to the problem \eqref{stokesp1c} and then employing the elliptic estimates \eqref{i_linear_lame} of Lemma \ref{lame eq} with $r=2n-2j\ge 2$ to the resulting problems, using the notation \eqref{nota2}, we obtain
\begin{align}\label{fes3} 
 \norm{\dt^j  u  }_{2n-2j}^2 & \ls
\norm{ \dt^{j+1} u   }_{2n-2j-2 }^2+\norm{ \nabla^2 \dt^{j} \eta}_{2n-2j-2 }^2+ \norm{ \dt^j G   }_{2n-2j-2}^2\nonumber
  \\&\ls   \norm{\dt^{j+1} u   }_{2n-2(j+1) }^2 +\norm{   \dt^{j} \eta}_{2n-2j  }^2  + \mathcal{X}_n. 
\end{align}
A simple induction on \eqref{fes3} yields, since $\dt\eta=u$,
\begin{align}\label{fes4} 
\sum_{j=0}^{n}\norm{ \dt^j u  }_{2n-2j }^2
 &  \ls \ns{\dt^{n} u}_{0} +\sum_{j=0}^{n-1}\norm{ \dt^j \eta  }_{2n-2j }^2  +\mathcal{X}_n\nonumber
\\&   \le \bar{\mathcal{E}}_n^t +\norm{  \eta  }_{2n  }^2 +\sum_{j=1}^{n-1}\norm{ \dt^{j-1} u  }_{2n-2j }^2    +\mathcal{X}_n\nonumber
\\&  \ls  \sum_{j=0}^{n-2}\norm{ \dt^{j} u  }_{2n-2j-2}^2     +\bar{\mathcal{E}}_n^t  + \bar{\mathcal{E}}_n^\sharp  +\mathfrak{E}_n  + \mathcal{X}_n. 
\end{align}
Here we have used the fact $\norm{  \eta  }_{2n+1}^2\ls \bar{\mathcal{E}}_n^\sharp  +\mathfrak{E}_n$ by Poincar\'e's inequality.
Using the Sobolev interpolation and Young's inequality, we can improve \eqref{fes4} to be
\begin{equation}
  \sum_{j=0}^{n}\norm{ \dt^j u  }_{2n-2j }^2
 \ls  \sum_{j=0}^{n-2}\norm{ \dt^{j} u  }_{0}^2   + \bar{\mathcal{E}}_n^t +\mathfrak{E}_n + \bar{\mathcal{E}}_n^\sharp  + \mathcal{X}_n
\ls   \bar{\mathcal{E}}_n^t +\bar{\mathcal{E}}_n^\sharp  +\mathfrak{E}_n +  \mathcal{X}_n.
 \end{equation}
Using \eqref{p_G_e_0c} to estimate $\mathcal{X}_n\ls ({ \se{n}  })^2$, we then conclude \eqref{e2nc}.
\end{proof}

\subsubsection{Synthesis}

We now chain all the estimates derived previously to conclude the following.
\begin{prop}
For $n=N+2$ or $2N$, there exists an energy $\tilde{ \mathcal{E}}_{n}$ which is equivalent to $\mathcal{E}_{n}$ such that
\begin{equation}\label{sys2nc}
\frac{d}{dt}\tilde{ \mathcal{E}}_{2N}+ {\mathcal{D}}_{2N}  \ls     \sqrt{ \se{N+2}  }{\mathcal{E}}_{2N}
\end{equation}
and
\begin{equation}\label{sysn+2c}
\frac{d}{dt}\tilde{ \mathcal{E}}_{N+2}+ {\mathcal{D}}_{N+2}  \le 0.
\end{equation}
\end{prop}
\begin{proof}
We first deduce from \eqref{qqest} of Proposition \ref{qqestprop} and \eqref{d2nc} of Proposition \ref{p_upper_evolution  N'c} that
 \begin{equation}\label{fes5}
 \dtt\mathfrak{E}_n  + \mathcal{D}_n
 \ls { \se{n}  }\sd{n}+\norm{ u  }_{2n-1}^2
+\bar{\mathcal{D}}_n^t+\bar{\mathcal{D}}_n^\ast+\bar{\mathcal{D}}_n^\sharp.
\end{equation}
We may use the Sobolev interpolation and Young's inequality to improve \eqref{fes5} to be
 \begin{equation}\label{fes6}
 \dtt\mathfrak{E}_n  + \mathcal{D}_n
 \ls { \se{n}  }\sd{n}+\norm{ u  }_{0}^2
+\bar{\mathcal{D}}_n^t+\bar{\mathcal{D}}_n^\ast+\bar{\mathcal{D}}_n^\sharp \ls { \se{n}  }\sd{n}
+\bar{\mathcal{D}}_n^t+\bar{\mathcal{D}}_n^\ast+\bar{\mathcal{D}}_n^\sharp.
\end{equation}

We now let $n$ denote either $2N$ or $N+2$ through the proof, and we use the compact notation
\begin{equation}
\mathcal{Z}_n\text{ with }\mathcal{Z}_{2N}:=\sqrt{ \se{2N}  }\sd{ 2N}+\sqrt{ \se{N+2}  }\se{ 2N} \text{ and }\mathcal{Z}_{N+2}:=\sqrt{ \se{2N}  }\sd{ N+2} .
\end{equation}
We then deduce from Propositions \ref{i_temporal_evolution  Nc}, \ref{p_upper_evolution  N12c}, \ref{p_upper_evolution  N'132c} and \eqref{fes6} that for $0<\epsilon\ll 1$,
\begin{align} \label{i_te_02nc} 
 & \frac{d}{dt}\left(\bar{\mathcal{E}}_{n}^t+ \bar{\mathcal{E}}_{n}^*+\epsilon\left( \bar{\mathcal{E}}_{n}^\sharp+2\sum_{\substack{\al\in \mathbb{N}^2\\  |\al|\le 2n}}\int_\Omega \bar\rho \partial^\alpha     u  \cdot \partial^\alpha  \eta\right)+\epsilon^2 \mathfrak{E}_{n} \right)
+ \bar{\mathcal{D}}_{n}^t+\bar{\mathcal{D}}_{n}^*+\epsilon \bar{\mathcal{D}}_{n}^\sharp+\epsilon^2  {\mathcal{D}}_{n}\nonumber
 \\&\quad \ls   \mathcal{Z}_n+ \epsilon(\bar{\mathcal{D}}_{n}^t+\bar{\mathcal{D}}_{n}^{\ast} )+\epsilon^2(\bar{\mathcal{D}}_{n}^t+\bar{\mathcal{D}}_{n}^*+ \bar{\mathcal{D}}_{n}^\sharp). 
\end{align}
Taking $\epsilon>0$ sufficiently small, we obtain
\begin{equation} \label{i_te_02n1c}
\begin{split}
 \frac{d}{dt}\left(\bar{\mathcal{E}}_{n}^t+ \bar{\mathcal{E}}_{n}^*+\epsilon\left( \bar{\mathcal{E}}_{n}^\sharp+2\sum_{\substack{\al\in \mathbb{N}^2\\  |\al|\le 2n}}\int_\Omega \bar\rho \partial^\alpha     u  \cdot \partial^\alpha  \eta\right)+\epsilon^2 \mathfrak{E}_{n} \right)
+ \bar{\mathcal{D}}_{n}^t+\bar{\mathcal{D}}_{n}^*+\epsilon \bar{\mathcal{D}}_{n}^\sharp+\epsilon^2  {\mathcal{D}}_{n}
 \ls   \mathcal{Z}_n .
  \end{split}
\end{equation}

We now define
\begin{equation}
\tilde{ \mathcal{E}}_{n}:=\bar{\mathcal{E}}_{n}^t+ \bar{\mathcal{E}}_{n}^*+\epsilon\left( \bar{\mathcal{E}}_{n}^\sharp+2\sum_{\substack{\al\in \mathbb{N}^2\\  |\al|\le 2n}}\int_\Omega \bar\rho \partial^\alpha     u  \cdot \partial^\alpha  \eta\right)+\epsilon^2 \mathfrak{E}_{n}.
\end{equation}
By Proposition \ref{e2ncprop}, we know that for fixed sufficiently small $\epsilon>0$,
\begin{equation}
\mathcal{E}_{n}\ls \tilde{ \mathcal{E}}_{n}+(\mathcal{E}_{n} )^{2},
\end{equation}
which implies that $\tilde{ \mathcal{E}}_{n}$ is equivalent to $ \mathcal{E}_{n}$ since ${\mathcal{E}}_{2N} (T)\le \delta$ is small. We thus deduce \eqref{sys2nc} and \eqref{sysn+2c} from \eqref{i_te_02n1c} by recalling the notation $\mathcal{Z}_n$ and using again that ${\mathcal{E}}_{2N} (T)\le \delta$ is small.
\end{proof}

\subsection{Global energy estimates}

In this subsection, we shall conclude our global energy estimates of the solution to \eqref{reformulationc}.

 We first show the boundedness of $\mathcal{E}_{2N} +\int_0^t\mathcal{D}_{2N} $.

\begin{prop} \label{Dglec}
There exists a universal constant $0<\delta<1$ so that if $\mathcal{G}_{2N}(T)\le\delta$, then
\begin{equation}\label{Dgc}
\mathcal{E}_{2N} (t)+\int_0^t\mathcal{D}_{2N} \lesssim
\mathcal{E}_{2N} (0)    \text{ for all
}0\le t\le T.
\end{equation}
\end{prop}
\begin{proof}
Integrating \eqref{sys2nc} directly in time, we find that
\begin{align} 
    {\mathcal{E}}_{2N}
+ \int_0^t{\mathcal{D}}_{2N}
&\ls {\mathcal{E}}_{2N}(0)+\int_0^t \sqrt{ \se{N+2}  }  \mathcal{E}_{2N}\nonumber
\\& \ls {\mathcal{E}}_{2N}(0)+\sup_{0\le r\le t}\mathcal{E}_{2N}(r)
 \int_0^t \sqrt{\delta}(1+r)^{-N+2}dr\nonumber
 \\& \ls {\mathcal{E}}_{2N}(0)+\sqrt{\delta}\sup_{0\le r\le t}\mathcal{E}_{2N}(r).
\end{align}
Here we have used the fact that $N\ge 4$. This proves the estimate \eqref{Dgc} since $\delta$ is small.
\end{proof}

It remains to show the decay estimates of
$\mathcal{E}_{N+2} $.

\begin{prop} \label{decaylmc}
There exists a universal constant $0<\delta<1$ so that if $\mathcal{G}_{2N}(T)\le\delta$, then
\begin{equation}\label{n+2c}
(1+t)^{2N-4} \mathcal{E}_{N+2} (t)\lesssim
\mathcal{E}_{2N} (0)  \ \text{for all
}0\le t\le T.
\end{equation}
\end{prop}
\begin{proof}
We will use \eqref{sysn+2c} to derive the decay estimates. For this, we shall estimate $\mathcal{E}_{N+2} $ in terms of $\mathcal{D}_{N+2} $. Notice that ${\mathcal{D}}_{N+2} $ can control every term in ${\mathcal{E}}_{N+2} $ except $ \ns{\nabla_\ast^{2(N+2)+1}\eta}_{0}$. The key point is to use the Sobolev interpolation as \cite{RG,GT_per}. Indeed, we first have that
\begin{align} \label{intep0c} 
\ns{\nabla_\ast^{2(N+2)+1}\eta}_{0}&\le  \norm{\nabla_\ast^{2(N+2)}\eta}_{0} ^{2\theta} \norm{\nabla_\ast^{4N+1}\eta}_{0}^{2(1-\theta)}\nonumber
\\&\le ( {\mathcal{D}_{N+2} })^\theta({\mathcal{E}_{2N} })^{1-\theta},\text{
where }\theta=\frac{2N-4}{2N-3}. 
\end{align}
Hence, we may deduce
\begin{equation} \label{intepc}
{\mathcal{E}}_{N+2} \le({\mathcal{D}}_{N+2} )^\theta({\mathcal{E}}_{2N} )^{1-\theta}.
\end{equation}

Now since by Proposition \ref{Dglec},
\begin{equation}
\sup_{0\le r\le t}\mathcal{E}_{2N} (r)\lesssim
\mathcal{E}_{2N} (0) :=\mathcal{M}_0,
\end{equation}
we obtain from  \eqref{intepc} that
\begin{equation} \label{u2c}
\tilde{\mathcal{E}}_{N+2} \lesssim\mathcal{E}_{N+2} \lesssim\mathcal{M}_0^{1-\theta}
(\mathcal{D}_{N+2} )^\theta.
\end{equation}
Hence by \eqref{sysn+2c} and \eqref{u2c}, there exists some constant $C>0$
such that
\begin{equation}
\frac{d}{dt} \tilde{\mathcal{E}}_{N+2} +\frac{C}{\mathcal{M}_0^s}
(\tilde{\mathcal{E}}_{N+2} )^{1+s}\le 0,\ \text{ where } s =
\frac{1}{\theta}-1 = \frac{1}{2N-4}.
\end{equation}
Solving this differential inequality directly, we obtain
\begin{equation} \label{u3c}
\mathcal{E}_{N+2} (t)\ls \tilde{\mathcal{E}}_{N+2} (t)\ls \frac{\mathcal{M}_0}{(\mathcal{M}_0^s
+ s C( \mathcal{E}_{N+2} (0))^s t)^{1/s} }
{\mathcal{E}}_{N+2} (0).
\end{equation}
Using that ${\mathcal{E}}_{N+2} (0)\lesssim\mathcal{M}_0 $ and
the fact $1/s=2N-4>1$, we obtain from \eqref{u3c} that
\begin{equation}
{\mathcal{E}}_{N+2} (t)\lesssim
\frac{\mathcal{M}_0}{(1+sCt)^{1/s} }\lesssim
\frac{\mathcal{M}_0}{(1+t^{1/s}) } =
\frac{\mathcal{M}_0}{(1+t^{2N-4}) }.
\end{equation}
This directly implies \eqref{n+2c}.
\end{proof}

Now we can arrive at our ultimate energy estimates for
$\mathcal{G}_{2N} $.
\begin{thm}\label{Apc}
There exists a universal $0 < \delta < 1$ so that if $
\mathcal{G}_{2N} (T) \le \delta$, then
\begin{equation}\label{Aprioric}
\mathcal{G}_{2N} (t) \ls\mathcal{E}_{2N} (0) \text{ for all }0 \le t \le
 T.
\end{equation}
\end{thm}
\begin{proof}
The conclusion follows directly from the definition of $\mathcal{G}_{2N}$ and Propositions
\ref{Dglec}--\ref{decaylmc}.
\end{proof}

\section{Incompressible MHD system}\label{sec incom}

In this section, we will derive the a priori estimates for the smooth solution $(\eta,u,p)$ to the incompressible MHD system \eqref{reformulationic}. We assume throughout the section that the solution obeys the estimate $\mathcal{G}_{2N}(T) \le \delta$ for sufficiently small $\delta>0$.

\subsection{Energy evolution}

In this subsection we derive energy evolution estimates for temporal and horizontal spatial derivatives by using the energy-dissipation structure of the system.

\subsubsection{Energy evolution of time derivatives}

For the temporal derivatives, it is a key to use the following geometric formulation. As well explained by \cite{GT_per} in the study of the incompressible viscous surface wave problem, the reason is that if we attempted to use the linear perturbed formulation \eqref{perturb}, we would be unable to control the interaction between $\dt^np$ and $\diverge \dt^n u$. Applying $\dt^j$ for $j=0,\dots,n$ to the system \eqref{reformulationic}, we find that
\begin{equation}\label{linear_geometric}
\begin{cases}
  \dt (\dt^j \eta) =\dt^j u  & \text{in } \Omega
\\  \rho_0\dt (\dt^j u) -\mu \da (\dt^j u)+\naba (\dt^jp)- \kappa\bar b^2\p_{3}^2(\dt^j \eta) = F^{1,j}  & \text{in } \Omega \\
 \diva (\dt^j u) = F^{2,j} & \text{in } \Omega\\
\dt^j u  =0 & \text{on } \p\Omega,
\end{cases}
\end{equation}
where for $i=1,2,3,$
\begin{equation}\label{F_def_start}
\begin{split}
 F_i^{1,j}  = \sum_{0 < \ell \le j}  C_j^\ell\left\{\mu
\mathcal{A}_{lk} \p_k ( \dt^\ell  \mathcal{A}_{lm} \dt^{j-\ell}\p_m u_i)
   +  \mu \dt^\ell \mathcal{A}_{lk}\dt^{j-\ell}\p_k ( \mathcal{A}_{lm}\p_m u_i) -  \dt^\ell  \mathcal{A}_{ik} \dt^{j-\ell}\p_k p\right\}.
 \end{split}
\end{equation}
Since we can not hope to get any estimates of $p$ ($i.e.$ $\partial_t^j p$) without spatial derivatives, we have to pay more attention on the expression of $F^{2,j}$.  We will need some structural
conditions on $F^{2,j}$ which allow us to integrate by parts in the interaction between $\partial_t^j p$ and $F^{2,j}$. Indeed, note that $ \diva u=\p_i(\a_{mi}u_m)=0$. Then we have
\begin{equation}\label{Qstr}
F^{2,j}= {\rm div} Q^{2,j}\text{ with }Q_i^{2,j}=
 -\sum_{0<\ell\le j}C_j^\ell\partial_t^\ell \mathcal{A}_{mi} \partial_t^{j-\ell}u_m,\ i=1,2,3.
 \end{equation}
Since $u=0$ on $\p\Omega$, we have
\begin{equation}\label{Qstr1}
  Q^{2,j}=0\hbox{ on }\p\Omega.\end{equation}
These facts are important for handling the pressure term.

We record the estimates of these nonlinear terms $F^{1,j}$ and $ Q^{2,j}$ in the following lemma.
\begin{lem}\label{p_F_estimates}
For $n=N+2$ or $n=2N$, it holds that
\begin{equation}\label{p_F_e_01}
 \ns{F^{1,j} }_{0}+\ns{    Q^{2,j} }_{0}  + \ns{ \dt  Q^{2,n} }_{0}\ls \se{N+2} \sd{n}
\end{equation}
and
\begin{equation}\label{p_F_e_02}
\norm{ Q^{2,n} }_0^2   \ls \se{N+2}\se{n} .
\end{equation}
\end{lem}
\begin{proof}
Note that all terms in the definitions of $F^{1,j}$ and $ Q^{2,j}$ (and so $\dt  Q^{2,j}$) are at least quadratic; each term can be written in the form $X Y$, where $X$ involves fewer derivative counts than $Y$. We may use the usual Sobolev embeddings along with the definitions of $\se{n}$ and $\sd{n} $ to estimate $\norm{X}_{L^\infty}^2\ls  \se{N+2} $ and $\norm{Y}_{0}^2\ls   \sd{n} $.
Then $\norm{XY}_0^2\le \norm{X}_{L^\infty}^2\norm{Y}_{0}^2\ls   \se{N+2} \sd{n} $, and the estimate \eqref{p_F_e_01} follows. The estimate \eqref{p_F_e_02} follows similarly.
\end{proof}

For a generic integer $n\ge 3$, we define the temporal energy by
\begin{equation}
 \bar{\mathcal{E}}^{t}_{n} =
 \sum_{j=0}^{n} \left( \norm{\sqrt{\rho_0}\dt^j u}_0^2+\kappa\bar b^2 \norm{\p_3\dt^j\eta}_0^2
 \right)
\end{equation}
and the temporal dissipation by
\begin{equation}
 \bar{\mathcal{D}}_n^{t} = \sum_{j=0}^{ n}  \norm{\dt^j u}_1^2.
\end{equation}
Then we have the following energy evolution.
\begin{prop}\label{i_temporal_evolution  N}
For $n=N+2$ or $n=2N$, it holds that
\begin{equation} \label{i_te_0}
  \frac{d}{dt}\left(\bar{\mathcal{E}}_{n}^{t}+\int_\Omega  \nabla \dt^{n-1} p   Q^{2,n}\right)
+ \bar{\mathcal{D}}_{n}^{t}
  \ls   \sqrt{\se{N+2} } \sd{ n}.
\end{equation}
\end{prop}

\begin{proof}
We let $n$ denote either $2N$ or $N+2$ throughout the proof. Taking the dot product of the second equation of $\eqref{linear_geometric}$ with $\dt^j u$, $j=0,\dots,n,$ and then integrating by parts, using the third and fourth equations, we obtain
\begin{align}\label{i_ge_ev_0} 
&\hal  \frac{d}{dt} \int_\Omega  \rho_0 \abs{\dt^j u}^2
+  \mu \int_\Omega \abs{ \naba\dt^j u}^2+\kappa\bar b^2\int_\Omega  \p_{3} \dt^j\eta\cdot\p_{3} \dt^j u\nonumber
\\&\quad= \int_\Omega   (  \dt^j u\cdot F^{1,j}+  \dt^j p  \diverge_\a (\dt^j u))
 = \int_\Omega   (   \dt^j u\cdot F^{1,j}+  \dt^j p  F^{2,j}). 
\end{align}
By the first equation, we get
\begin{equation}\label{i_te_2}
 \int_\Omega  \p_3 \dt^j\eta\cdot \p_3 \dt^j u =  \int_\Omega  \p_3 \dt^j\eta\cdot \p_3 \dt^{j+1} \eta=\hal\frac{d}{dt}\int_\Omega \abs{\p_3\dt^j\eta}^2.
\end{equation}

We now estimate the right hand side of \eqref{i_ge_ev_0}. For the $F^{1,j}$ term, by \eqref{p_F_e_01}, we may bound
\begin{equation}
\int_\Omega  \dt^j u\cdot F^{1,j} \le   \norm{\dt^j u}_{0}   \norm{F^{1,j}}_0 \ls  \sqrt{\sd{n} } \sqrt{\se{N+2} \sd{n}}
 .
\end{equation}
For the $F^{2,j}$ term, we need much more care. First, since we can not get any estimates of $\partial_t^j p$ without spatial derivatives, we need to use the structure of $F^{2,j}$ and employ an integration by parts in space. Indeed, by \eqref{Qstr} and \eqref{Qstr1}, we deduce
\begin{equation}
\int_\Omega  \dt^j p  F^{2,j} =\int_\Omega  \dt^j p  {\rm div} Q^{2,j}
=-\int_\Omega  \nabla \dt^j p    Q^{2,j}
\end{equation}
Second, we need to consider the case $j< n$ and $j= n$ separately. For $j< n$, by \eqref{p_F_e_01} we have
\begin{equation}\label{i_te_4}
-\int_\Omega  \nabla \dt^j p   Q^{2,j}   \le   \norm{\nabla \dt^j p}_{0}   \norm{ Q^{2,j} }_0 \ls  \sqrt{\sd{n}} \sqrt{\se{N+2} \sd{n}}.
\end{equation}
The case $j= n$ is much more involved since we can not control $\nabla\dt^{ n}p$. We are then forced to integrate by parts in time:
\begin{equation}
 -\int_\Omega   \nabla \dt^n p   Q^{2,n}= -\frac{d}{dt}\int_\Omega  \nabla \dt^{n-1} p   Q^{2,n}+\int_\Omega     \nabla \dt^{n-1} p \dt  Q^{2,n}   .
\end{equation}
By \eqref{p_F_e_01}, we may bound
\begin{equation}\label{i_te_5}
 \int_\Omega  \nabla \dt^{n-1} p \dt  Q^{2,n} \ls   \norm{\nabla \dt^{n-1} p }_{0} \norm{\dt    Q^{2,n}}_{0} \ls   \sqrt{\sd{n}} \sqrt{\se{N+2} \sd{n}} .
\end{equation}

Now we combine \eqref{i_te_2}--\eqref{i_te_5} to deduce from \eqref{i_ge_ev_0} that, summing over $j$,
\begin{equation} \label{iiii}
  \frac{d}{dt}\left(\bar{\mathcal{E}}_{n}^{t}+\int_\Omega  \nabla \dt^{n-1} p   Q^{2,n}\right)
+  \sum_{j=0}^{ n}  \int_\Omega \mu  \abs{ \naba \dt^j u}^2
  \ls   \sqrt{\se{N+2} } \sd{ n}.
\end{equation}
We then seek to replace $ \abs{ \nabla_{\mathcal{A}} \dt^{j} u}^2$ with $\abs{\nabla \dt^{j} u}^2$
in \eqref{iiii}.  To this end we write
\begin{equation}\label{i_te_8}
 \abs{ \naba \dt^{j} u}^2 = \abs{\nabla \dt^{j} u}^2 +   \left(\naba \dt^{j} u + \nabla\dt^{j} u\right): \left(\naba \dt^{j} u - \nabla \dt^{j} u\right)
\end{equation}
and note that
\begin{equation}\label{i_te_9}
 \int_\Omega \left(\naba \dt^{j} u + \nabla\dt^{j} u\right): \left(\naba \dt^{j} u - \nabla \dt^{j} u\right)
\ls  (1 + \sqrt{\se{N+2}})\sqrt{\se{N+2}} \int_\Omega \abs{\nab \dt^{j} u}^2  \ls  \sqrt{\se{N+2}} \sd{n}.
\end{equation}
We may then use \eqref{i_te_8}--\eqref{i_te_9} to replace in \eqref{iiii} and derive \eqref{i_te_0}, by Poincar\'e's inequality.
\end{proof}

\subsubsection{Energy evolution of horizontal derivatives}

For the horizontal spatial derivatives, we shall use the following linear perturbed formulation
\begin{equation}\label{perturb}
\begin{cases}
 \dt\eta=u&\hbox{in }\Omega
\\\rho_0\partial_t u-\mu\Delta u+\nabla p- \kappa\bar b^2 \p_3^2\eta=G^1\quad&\hbox{in }\Omega
\\ \diverge u=G^2&\hbox{in }\Omega
\\ u  =0 &\hbox{on }\p\Omega,
\end{cases}
\end{equation}
where for $i=1,2,3$,
\begin{equation}\label{G1_def}
 G^1_i=   \mu\p_k((\a_{jk}\a_{jl}-\delta_{jk}\delta_{jl})\p_l u_i)-(\a_{ik}-\delta_{ik})\p_kp
 \end{equation}
 and
\begin{equation}
G^2=(\a_{jk}-\delta_{jk})\p_k u_j.\label{G2}
 \end{equation}

We record the estimates of the nonlinear terms $G^1$ and $G^2$ in the following lemma.
\begin{lem}\label{p_G_estimates}
For $n=N+2$ or $n=2N$, it holds that
 \begin{equation}\label{p_G_e_0}
  \ns{ \bar{\nab}_0^{2n-2} G^1}_{0}  +  \ns{ \bar{\nab}_0^{2n-2}  G^2}_{1}     \ls  \se{N+2}\se{n} .
 \end{equation}
We also have
\begin{equation}\label{p_G_e_001}
\ns{ \bar{\nab}_0^{ 4N-1} G^1}_{0} +  \ns{ \bar{\nab}_0^{ 4N-1}  G^2}_{1}      \ls  \se{N+2}(\sd{2N}+\mathcal{J}_{2N}+ \f )
\end{equation}
and
\begin{equation}\label{p_G_e_002}
\ns{ \bar{\nab}_0^{ 2(N+2)-1} G^1}_{1} +  \ns{ \bar{\nab}_0^{ 2(N+2)-1}  G^2}_{2}  \ls \se{2N}\sd{N+2} .
\end{equation}
\end{lem}
\begin{proof}
Note that all terms in the definitions of $G^i$ are at least quadratic. We apply these
space-time differential operators to $G^i$ and then expand using the Leibniz rule; each product in
the resulting sum is also at least quadratic. Then the estimate \eqref{p_G_e_0} follows in the same way as Lemma \ref{p_F_estimates}.

The last two terms in the right hand side of \eqref{p_G_e_001} is due to the control of the highest spatial derivatives in some products, which is not controlled by $\sd{2N}$ but $\mathcal{J}_{2N}+ \f$. It is crucial that the other factors in such products are of low derivatives and hence can be easily controlled by $\se{N+2}$. Then the estimate \eqref{p_G_e_001} follows.

The proof of the estimate \eqref{p_G_e_002} is somewhat easier. Indeed, we may write each term in the form $X Y$, where $X$ involves fewer derivative counts than $Y$; then we simply bound the various norms of $Y$ by $ \se{2N} $ and bound the various norms of $X$ by $\sd{N+2}$. Then the estimate \eqref{p_G_e_002} follows.
\end{proof}

For a generic integer $n\ge 3$, we define the  horizontal energy by
\begin{equation}
 \bar{\mathcal{E}}^{\ast}_{n} =   \norm{\sqrt{\rho_0 }\nabla_\ast  u }_{0,2n-1}^2  + \kappa\bar b^2\norm{\nabla_\ast\p_3  \eta }_{0,2n-1}^2
\end{equation}
and the horizontal dissipation by
\begin{equation}
 \bar{\mathcal{D}}_n^{\ast} = \norm{ \nabla_\ast  u }_{1,2n-1}^2 .
\end{equation}
Then we have the following energy evolution.

\begin{prop}\label{p_upper_evolution  N12}
It holds that
\begin{equation}\label{p_u_e_00}
\frac{d}{dt}\bar{\mathcal{E}}^{\ast}_{2N}+\bar{\mathcal{D}}_{2N}^{\ast}\ls\sqrt{ \se{N+2}  } ( \sd{2N}+\mathcal{J}_{2N} +\f)+\norm{\dt u}_{4N-2}^2
\end{equation}
and
\begin{equation}\label{p_u_e_00''}
\frac{d}{dt}\bar{\mathcal{E}}^\ast_{N+2}+\bar{\mathcal{D}}_{N+2}^\ast\ls\sqrt{ \se{2N}   } \sd{N+2}+\norm{\dt u}_{2(N+2)-2}^2  .
\end{equation}
\end{prop}
\begin{proof}
We let $n$ denote either $2N$ or $N+2$ throughout the proof. We take $\alpha\in \mathbb{N}^{2}$ so that $1\le |\al|\le 2n$. Applying $\p^\al $ to the second equation of \eqref{perturb} and then taking the dot product with $\p^\al u$, using the other equations as in Proposition \ref{i_temporal_evolution  N}, we find that
\begin{align} \label{p_u_e_1110} 
 &\hal \frac{d}{dt}\left(  \int_\Omega \rho_0 \abs{\p^\al   u }^2  +\kappa\bar b^2 \abs{\p_3\p^\al   \eta }^2\right)+  \mu\int_\Omega  \abs{\nabla \p^\al u }^2\nonumber
 \\&\quad=-\sum_{\beta<\alpha}C_\alpha^\beta\int_\Omega \p^{\al-\beta} \rho_0 \p^{\beta}\partial_t  u \cdot\p^\al   u+ \int_\Omega   \p^\al  u  \cdot  \p^\al  G^1 +\int_\Omega   \p^\al  p  \p^\al  G^2. 
\end{align}

We first deal with the summation term. If $\beta=0$, then the term in the sum is
\begin{equation}
-\int_\Omega \p^{\al} \rho_0 \partial_t  u \cdot\p^\al   u.
\end{equation}
For $\abs{\alpha}=1$, we bound
\begin{equation}
-\int_\Omega \p^{\al} \rho_0 \partial_t  u \cdot\p^\al   u
 \ls\norm{\nabla \rho_0}_{L^\infty}  \norm{ \dt u}_{L^2}\norm{\p^{\al}  u}_{L^2}
 \ls  \norm{ \dt u}_{0}\norm{ \p^\al   u}_{0} .
\end{equation}
For $2\le\abs{\alpha}\le 2n$, then we write $\al=\gamma+\alpha-\gamma$ for some $\gamma\in \mathbb{N}^2$ with $\abs{\gamma}=1$; we can then integrate by parts to have
\begin{align} 
&
-\int_\Omega \p^{\al} \rho_0 \partial_t  u \cdot\p^\al   u=\int_\Omega \p^{\al-\gamma} \rho_0\left( \p^{\gamma}\dt u\cdot\p^\al   u+  \dt u\cdot\p^{\al+\gamma}   u\right)\nonumber
\\&\quad
\ls\norm{ \p^{\al-\gamma} \rho_0}_{L^2} \left( \norm{ \p^{\gamma}\dt u}_{L^4}\norm{\p^\al   u}_{L^4}+   \norm{\dt u}_{L^\infty}\norm{\p^{\al+\gamma}   u}_{L^2}\right)
\ls \norm{ \dt u}_{2}\norm{\p^\al   u}_{1}. 
\end{align}
In conclusion, we have
\begin{equation}\label{111}
-\int_\Omega \p^{\al} \rho_0 \partial_t  u \cdot\p^\al  u
\ls \norm{ \dt u}_{2}\norm{\p^\al   u}_{1}.
\end{equation}
Now if $1\le \abs{\beta}\le2$, then we have
\begin{equation}
-\int_\Omega \p^{\al-\beta} \rho_0 \p^{\beta}\partial_t  u \cdot\p^\al   u
  \ls\norm{ \p^{\al-\beta}\rho_0}_{L^2}   \norm{ \p^{\beta}\dt u}_{L^4}\norm{\p^\al  u}_{L^4}
  \ls \norm{ \dt u}_{3}\norm{\p^\al   u}_{1}.
\end{equation}
If $3\le \abs{\beta}\le \abs{\alpha}-1\le 2n-1$, then we write $\beta=\gamma+\beta-\gamma$ for some $\gamma\in \mathbb{N}^2$ with $\abs{\gamma}=1$. We can then integrate by parts to have
\begin{align}\label{222} 
 &- \int_\Omega \p^{\al-\beta} \rho_0 \p^{\beta}\partial_t u \cdot\p^\al  u=\int_\Omega \p^{\al-\beta+\gamma} \rho_0 \p^{\beta-\gamma}\partial_t u \cdot\p^\al  u
  +\p^{\al-\beta} \rho_0 \p^{\beta-\gamma}\partial_t u \cdot\p^{\al+\gamma}  u\nonumber
 \\&\quad\ls \norm{ \p^{\beta-\gamma}\dt u}_{L^2}\left(\norm{ \p^{\al-\beta+\gamma}\rho_0}_{L^4}  \norm{\p^\al   u}_{L^4}+\norm{ \p^{\al-\beta}\rho_0}_{L^\infty}  \norm{\p^{\al+\gamma}   u}_{L^2}\right)\nonumber
 \\ &\quad\ls\norm{ \nabla \rho_0}_{2n-2}  \norm{ \dt u}_{0,2n-2}\norm{\p^\al   u}_{1}\ls  \norm{ \dt u}_{2n-2}\norm{\p^\al   u}_{1}. 
\end{align}
Hence, in light of \eqref{111}--\eqref{222}, we deduce from \eqref{p_u_e_1110} that
\begin{equation}\label{i_de_142}
-\sum_{\beta<\alpha}C_\alpha^\beta\int_\Omega \p^{\al-\beta} \rho_0 \p^{\beta}\partial_t  u \cdot\p^\al   u\ls \norm{ \dt u}_{2n-2}\norm{\p^\al   u}_{1}.
\end{equation}

We now estimate the remaining two terms on the right hand side of \eqref{p_u_e_1110}. Since $\abs{\al}\ge 1$, we may write $\alpha = \gamma +(\alpha-\gamma)$ for some $\gamma \in \mathbb{N}^2$ with $\abs{\gamma}=1$. We first consider the case $n=2N$. We can then integrate by parts and use \eqref{p_G_e_001} to have
\begin{align}\label{i_de_4} 
  \int_\Omega     \p^\al   u \cdot   \p^\al  G^1  &=- \int_\Omega     \p^{\alpha+\gamma}   u \cdot   \p^{\alpha-\gamma}  G^1
 \le \norm{\p^{\alpha+\gamma}   u}_{0}  \norm{ \p^{\alpha-\gamma}   G^1 }_{0}\nonumber \\&
\le \norm{\p^{\alpha}   u}_{1}  \norm{    G^1 }_{4N-1 }
\ls \sqrt{ \sd{2N} } \sqrt{ \se{N+2}(\sd{2N}+\mathcal{J}_{2N} +\f)  } .
\end{align}
For the $G^2$ term we do not need to (and we can not) integrate by parts:
\begin{equation}\label{i_de_5}
\begin{split}
  \int_\Omega  \p^\al   p \p^\al  G^2
&\le \norm{\p^{\alpha-\gamma}\p^\gamma  p}_{0}  \norm{ \p^{\alpha} G^2 }_{0}
\le \norm{\p^{\alpha-\gamma}\p^\gamma p}_{0}  \norm{    G^2 }_{4N}
  \\&\ls \sqrt{\mathcal{J}_{2N} }\sqrt{ \se{N+2}(\sd{2N}+\mathcal{J}_{2N} +\f)  } .
\end{split}
\end{equation}
Hence, by \eqref{i_de_142}--\eqref{i_de_5}, we deduce from \eqref{p_u_e_1110} that for all $1\le \abs{\alpha} \le 4N$,
\begin{align} \label{7878} 
 &\hal \frac{d}{dt}\left(  \int_\Omega \rho_0 \abs{\p^\al  u }^2  + \kappa\bar b^2\abs{\p_3\partial^\alpha  \eta }^2\right)+   \mu\int_\Omega  \abs{\nabla \p^\al u }^2\nonumber
 \\&\quad\ls  \sqrt{ \se{N+2}  } ( \sd{2N}+\mathcal{J}_{2N} +\f)+\norm{\dt u}_{4N-2}\norm{\p^\al   u}_{1}. 
\end{align}
The estimate \eqref{p_u_e_00} then follows from \eqref{7878} by summing over such $\alpha$ and using Poincar\'e's and Cauchy's inequalities.

We now consider the case $n=N+2$. By \eqref{p_G_e_002}, we have
\begin{equation} \label{i_de_4'}
\begin{split}
 & \int_\Omega   \p^\al u  \cdot  \p^\al G^1 \le \norm{\p^\al u}_0\norm{G^1}_{2(N+2)}
  \ls \sqrt{ \sd{N+2} } \sqrt{ \se{2N}\sd{N+2} } .
\end{split}
\end{equation}
For the $G^2$ term we need a bit more care. If $1\le|\al|\le 2(N+2) -1$, then
\begin{align}\label{i_de_5'0} 
  \int_\Omega  \p^\al   p \p^\al   G^2
&\le \norm{\p^{\alpha-\gamma}\p^\gamma   p}_{0}  \norm{ \p^{\alpha}  G^2 }_{0}
\le \norm{\nabla_\ast  p}_{2(N+2) -2}  \norm{    G^2 }_{ 2(N+2)}\nonumber
  \\&\ls \sqrt{ \sd{N+2} }\sqrt{ \se{2N}\sd{N+2}}  . 
\end{align}
If $|\al|= 2(N+2) $,  then we may write $\alpha = \beta +\gamma+(\alpha-\beta-\gamma)$ for some $\beta,\gamma \in \mathbb{N}^2$ with $\abs{\beta}=\abs{\gamma}=1$. Then we integrate by parts to have
\begin{align}\label{i_de_5'} 
  \int_\Omega  \p^\al p \p^\al  G^2&=-\int_\Omega  \p^{\al-\beta-\gamma}\p^\gamma  p \p^{\al+\beta}  G^2
\le \norm{\nabla_\ast  p}_{2(N+2) -2}  \norm{  G^2 }_{2(N+2)+1 }\nonumber
  \\&\ls \sqrt{ \sd{N+2} }\sqrt{ \se{2N}\sd{N+2}}  . 
\end{align}
Hence, by \eqref{i_de_142} and \eqref{i_de_4'}--\eqref{i_de_5'}, we deduce from \eqref{p_u_e_1110} that for all $1\le\abs{\alpha} \le 2(N+2)$,
\begin{equation} \label{7878'}
\hal \frac{d}{dt}\left(  \int_\Omega \rho_0 \abs{\p^\al   u }^2  + \kappa\bar b^2\abs{\p_3\partial^\alpha   \eta }^2\right)+  \mu \int_\Omega  \abs{\nabla \p^\al  u }^2
\ls  \sqrt{ \se{2N}   } \sd{N+2}+\norm{\dt u}_{2(N+2)-2}\norm{\p^\al   u}_{1}.
\end{equation}
The estimate \eqref{p_u_e_00''} then follows from \eqref{7878'}.
\end{proof}

\subsubsection{Energy evolution controlling $\eta$}

Note that the dissipation estimates in $\bar{\mathcal{D}}_n^t$ of Proposition \ref{i_temporal_evolution  N} and $\bar{\mathcal{D}}_n^\ast$ of Proposition \ref{p_upper_evolution  N12} only contain $u$, in this subsubsection we will recover certain dissipation estimates of $\eta$. Since $\dt \eta =u$, we then only need to estimate for $\eta$ without time derivatives. The estimates result from testing the linear perturbed formulation \eqref{perturb} by $\eta$. Note that the boundary condition $\eta=0$ on $\p\Omega$ guarantees this. Moreover, the Jacobian identity $J=1$ gives the control of ${\rm div}\eta$;
indeed,
\begin{equation}\label{phhicon}
{\rm div}\eta=\Phi,\quad \Phi=-({\rm det} (I+\nabla\eta)-1-{\rm div}\eta)=O(\nabla\eta\nabla^2\eta).
\end{equation}
Again, as for $F^{2,j}$, we will need some structural
conditions on $\Phi$ which allow us to integrate by parts in the interaction between $p$ and $\Phi$ (without spatial derivatives). This is not apparent, we need to do some lengthy but straightforward computations. Indeed, we expand the expression of $\Phi$ to conclude that
 \begin{equation}\label{Qstr2}
\Phi= {\rm div}\Psi \text{ with }
\Psi= \left(\begin{array}{ccc}\eta_1\p_2\eta_2+\eta_1\p_3\eta_3+\eta_1(\p_2\eta_2\p_3\eta_3-\p_3\eta_2\p_2\eta_3)
\\-\eta_1\p_1\eta_2+\eta_2\p_3\eta_3+\eta_1(\p_3\eta_2\p_1\eta_3-\p_1\eta_2\p_3\eta_3)
\\-\eta_1\p_1\eta_3-\eta_2\p_2\eta_3+\eta_1(\p_1\eta_2\p_2\eta_3-\p_2\eta_2\p_1\eta_3)
\end{array}\right).
 \end{equation}
Moreover, since $\eta=0$ on $\p\Omega$, we have
\begin{equation}\label{Qstr111}
 \Psi=0\hbox{ on }\p\Omega.
\end{equation}

We record some estimates of $\Phi$ and $\Psi$ in the following lemma.
\begin{lem}\label{p_G_estimates''}
It holds that
\begin{equation}\label{p_G_e_001''}
 \ns{ \Phi}_{4N} \ls  \se{N+2}(\sd{2N}+\mathcal{J}_{2N}+ \f ),
\end{equation}
\begin{equation}\label{p_G_e_002''}
 \ns{\Phi}_{2(N+2)+1} \ls \se{2N}\sd{N+2}
\end{equation}
and
\begin{equation}\label{Phe_0}
  \ns{ \Psi}_0 \ls \mathcal{E}_{3} \mathcal{D}_{3} .
\end{equation}
\end{lem}
\begin{proof}
The proof proceeds similarly as Lemma \ref{p_G_estimates}.
\end{proof}

For a generic integer $n\ge 3$, we define the recovering energy by
\begin{equation}
 \bar{\mathcal{E}}^\sharp_{n} = \mu  \norm{\nabla     \eta}_{0,2n}^2
\end{equation}
and the corresponding dissipation by
\begin{equation}
 \bar{\mathcal{D}}_n^\sharp =\norm{  \p_3 \eta}_{0,2n}^2+\norm{      \eta}_{0,2n}^2.
\end{equation}
Then we have the following energy evolution.

\begin{prop}\label{p_upper_evolution  N'132}
It holds that
\begin{align}\label{p_u_e_00'132} 
&\frac{d}{dt}\left(\bar{\mathcal{E}}^\sharp_{2N}+2\sum_{\substack{\al\in \mathbb{N}^2\\  |\al|\le 4N}}\int_\Omega \partial^\alpha ( \rho_0   u )\cdot \partial^\alpha  \eta\right)+\bar{\mathcal{D}}_{2N}^\sharp\nonumber
\\&\quad\ls \sqrt{ \se{N+2}  }( \sd{2N} + \mathcal{J}_{2N} +\f)+\bar{\mathcal{D}}_{2N}^t+\bar{\mathcal{D}}_{2N}^{\ast}+\norm{  u}_{4N-1}^2 
\end{align}
and
\begin{align}\label{p_u_e_00'12132} 
&\frac{d}{dt}\left(\bar{\mathcal{E}}^\sharp_{N+2}+2\sum_{\substack{\al\in \mathbb{N}^2\\  |\al|\le 2(N+2)}}\int_\Omega \partial^\alpha ( \rho_0   u )\cdot \partial^\alpha  \eta\right)+\bar{\mathcal{D}}_{N+2}^\sharp\nonumber
\\&\quad\ls\sqrt{ \se{2N}   } \sd{N+2} +\bar{\mathcal{D}}_{N+2}^t +\bar{\mathcal{D}}_{N+2}^{\ast}+\norm{  u}_{2(N+2)-1}^2. 
\end{align}
\end{prop}
\begin{proof}
We let $n$ denote either $2N$ or $N+2$ throughout the proof.
Applying $\partial^\alpha$ with $\al\in \mathbb{N}^{2}$ so that $ |\al|\le 2n$ to the second equation of \eqref{perturb} and then taking the dot product with $\p^\al \eta$, since $\eta=0$ on $\p\Omega$, by \eqref{phhicon}, we find that
\begin{align} \label{p_u_e_111'} 
 &\int_\Omega \partial^\alpha ( \rho_0 \dt    u) \cdot \partial^\alpha \eta
  +\frac{1}{2}\frac{d}{dt} \int_\Omega  \mu    \abs{\nabla  \p^\al  \eta}^2  +\kappa\bar b^2\int_{\Omega}  \abs{\p_3\partial^\alpha  \eta }^2\nonumber
 \\&\quad=  \int_\Omega   \p^\al \eta  \cdot  \p^\al G^1 +\int_\Omega   \p^\al p  \diverge \p^\al \eta=  \int_\Omega   \p^\al \eta  \cdot  \p^\al G^1 +\int_\Omega   \p^\al p  \p^\al \Phi. 
\end{align}
For the first term on the left hand side of \eqref{p_u_e_111'}, we integrate by parts in time to have
\begin{align}\label{ttte} 
\int_\Omega  \partial^\alpha ( \rho_0 \dt    u) \cdot \partial^\alpha \eta&=\frac{d}{dt}\int_\Omega  \partial^\alpha ( \rho_0      u) \cdot   \partial^\alpha  \eta-\int_\Omega   \partial^\alpha ( \rho_0      u) \cdot \dt \partial^\alpha  \eta\nonumber\\&=\frac{d}{dt}\int_\Omega  \partial^\alpha ( \rho_0      u) \cdot   \partial^\alpha  \eta-\int_\Omega   \partial^\alpha ( \rho_0      u) \cdot   \partial^\alpha  u. 
\end{align}
We estimate the last term in \eqref{ttte}. If $0\le \abs{\alpha}\le1$,  it is easy to bound that
\begin{equation}\label{eeee}
 -\int_\Omega  \partial^\alpha ( \rho_0      u) \cdot   \partial^\alpha  u\ls \norm{ u}_{0,1}^2.
\end{equation}
If $2\le\abs{\alpha}\le 2n$, we may write $\alpha = \beta +(\alpha-\beta)$ for some $\beta \in \mathbb{N}^2$ with $\abs{\beta}=1$; we can then integrate by parts and expand to have
\begin{align} 
 -\int_\Omega   \partial^\alpha ( \rho_0      u) \cdot   \partial^\alpha  u
& =\int_\Omega   \partial^{\alpha-\beta} ( \rho_0      u) \cdot   \partial^{\alpha+\beta}  u\nonumber
\\& =\sum_{\gamma\le \alpha-\beta}\int_\Omega C_{\alpha-\beta}^\gamma    \p^\gamma \rho_0     \partial^{\alpha-\beta-\gamma} u  \cdot   \partial^{\alpha+\beta}  u. 
\end{align}
For $0\le \abs{\gamma}\le 1$, then
\begin{equation}
\int_\Omega    \p^\gamma \rho_0     \partial^{\alpha-\beta-\gamma} u  \cdot   \partial^{\alpha+\beta}  u
\le \norm{\p^\gamma\rho_0}_{L^\infty}\norm{ \partial^{\alpha-\beta-\gamma} u}_{L^2} \norm{\partial^{\alpha+\beta}  u}_{L^2}\ls \norm{u}_{2n-1}\norm{u}_{0,2n+1}.
\end{equation}
For $2\le \abs{\gamma}\le \abs{\alpha-\beta}\le 2n-1$, then
\begin{equation}\label{ffff}
\int_\Omega    \p^\gamma \rho_0     \partial^{\alpha-\beta-\gamma} u  \cdot   \partial^{\alpha+\beta}  u
\le \norm{\p^\gamma\rho_0}_{L^2}\norm{ \partial^{\alpha-\beta-\gamma} u}_{L^\infty} \norm{\partial^{\alpha+\beta}  u}_{L^2}\ls \norm{u}_{2n-1}\norm{u}_{0,2n+1}.
\end{equation}
Hence, in light of \eqref{eeee}--\eqref{ffff}, we deduce that
\begin{equation}
 -\int_\Omega   \partial^\alpha ( \rho_0      u) \cdot   \partial^\alpha  u\ls\norm{u}_{2n-1}\norm{u}_{0,2n+1}.
\end{equation}

We now estimate the terms on the right hand side of \eqref{p_u_e_111'}. For $\al=0$, we easily have
\begin{equation}
   \int_\Omega  \eta  \cdot   G^1
    \ls \sqrt{ \sd{3} } \sqrt{ \se{3}\sd{3}  } .
\end{equation}
The pressure term is needed much more care; by \eqref{Qstr2}, \eqref{Qstr111} and \eqref{Phe_0}, we obtain
\begin{equation}
\begin{split}
 & \int_\Omega     p   \Phi= \int_\Omega     p   {\rm div}\Psi= -\int_\Omega  \nabla  p   \Psi\le \norm{\nabla p}_0\norm{\Psi}_0\ls \sqrt{ \sd{3} } \sqrt{ \se{3}\sd{3}  } .
\end{split}
\end{equation}
We then turn to the case $\al\neq 0$. We first consider the case $n=2N$. Similarly as \eqref{i_de_4},  we have
\begin{equation}
\int_\Omega   \p^\al \eta \cdot  \p^\al G^1    \ls  \norm{\p^{\alpha}  \eta}_{1}  \norm{   G^1 }_{4N-1}
 \ls \sqrt{ \f } \sqrt{ \se{N+2}(\sd{2N}+\mathcal{J}_{2N} +\f)  }.
\end{equation}
Similarly as \eqref{i_de_5}, by using instead \eqref{p_G_e_001''},
\begin{equation}
  \int_\Omega  \p^\al p \p^\al \Phi
\le \norm{\nabla_\ast p}_{4N-1}  \norm{    \Phi }_{4N}
 \ls \sqrt{ \mathcal{J}_{2N} }\sqrt{ \se{N+2}(\sd{2N}+\mathcal{J}_{2N} +\f)  } .
\end{equation}

We now consider the case $n=N+2$. By \eqref{p_G_e_002}, we have
\begin{equation}
\int_\Omega   \p^\al \eta \cdot  \p^\al G^1    \ls  \norm{\p^{\alpha}  \eta}_{0}  \norm{   G^1 }_{2(N+2)}
 \ls \sqrt{ \sd{N+2}} \sqrt{ \se{2N}\sd{N+2} }.
\end{equation}
For the pressure term, similarly as \eqref{i_de_5'0}--\eqref{i_de_5'}, by using instead \eqref{p_G_e_002''}, we obtain
\begin{equation}
  \int_\Omega  \p^\al p \p^\al G^2
\ls \norm{ \nabla_\ast p}_{2(N+2)-2}  \norm{\Phi }_{2(N+2)+1}
 \ls \sqrt{ \sd{N+2} }\sqrt{ \se{2N}\sd{N+2}}  .
\end{equation}

Consequently, the estimates \eqref{p_u_e_00'132} and \eqref{p_u_e_00'12132} follow by collecting the estimates, summing over such $\alpha$ and using Poincar\'e's and Cauchy's inequalities.
\end{proof}

\subsection{Estimates via Stokes regularity}

We now apply the elliptic regularity theory of certain Stokes problems to improve the energy-dissipation estimate with the energy evolution in hand.

\subsubsection{Dissipation improvement}

We first consider the improvement of the dissipation estimates; the energy estimates of $\eta$ will be improved along the way.
\begin{prop}\label{p_upper_evolution  N'}
For $n\ge 3$, there exists an energy $\mathfrak{E}_n $ which is equivalent to $\norm{\eta}_{2n}^2$ such that
\begin{equation}\label{d2n}
\frac{d}{dt}\mathfrak{E}_{2N}+\mathcal{D}_{2N}\ls  { \se{N+2}  }(\sd{2N} +  \mathcal{J}_{2N} +\f)+\bar{\mathcal{D}}_{2N}^t+\bar{\mathcal{D}}_{2N}^\ast+\bar{\mathcal{D}}_{2N}^\sharp
\end{equation}
and
\begin{equation}\label{d2n+2}
\frac{d}{dt}\mathfrak{E}_{N+2} +\mathcal{D}_{N+2}\ls { \se{2N}   } \sd{N+2}+\bar{\mathcal{D}}_{N+2}^t+\bar{\mathcal{D}}_{N+2}^\ast+\bar{\mathcal{D}}_{N+2}^\sharp.
\end{equation}
\end{prop}
\begin{proof}
We let $n$ denote either $2N$ or $N+2$ throughout the proof, and we compactly write
\begin{equation}\label{n1}
 \mathcal{Y}_n =  \ns{ \bar{\nab}^{2n-1} G^1}_{0}  +  \ns{ \bar{\nab}^{2n-1}  G^2}_{1} +  \ns{   \Phi}_{2n-1}  .
\end{equation}

We divide the proof into several steps.

{\bf Control terms with time derivatives}

Applying the time derivatives $\dt^{j},\ j=1,\dots,n-1$ to the equations \eqref{perturb}, we find that
\begin{equation}\label{stokesp1}
\begin{cases}
-\mu\Delta \dt^{j} u+\nabla \dt^{j} p =-\rho_0\dt^{j+1} u+\kappa\bar b^2 \p_{3}^2 \dt^{j}\eta+\dt^{j} G^1& \text{in }
\Omega
\\ \diverge \dt^{j} u=\dt^{j} G^2  & \text{in }
\Omega
\\ \dt^{j} u=0 &\text{on }\p\Omega.
\end{cases}
\end{equation}
Applying the elliptic estimates \eqref{stokes es} of Lemma \ref{i_linear_elliptic2} with $r=2n-2j+1\ge 3$ to the problem \eqref{stokesp1} and using the notation \eqref{n1}, we obtain
\begin{align}\label{ffes1} 
 &\norm{\dt^j  u  }_{2n-2j+1}^2 + \norm{\nab \dt^j  p  }_{2n-2j-1}^2\nonumber
 \\  &\quad\ls
\norm{\rho_0\dt^{j+1} u   }_{2n-2j-1 }^2+\norm{ \p_{3}^2 \dt^{j }\eta}_{2n-2j-1 }^2+ \norm{ \dt^j G^1   }_{2n-2j-1}^2
+ \norm{\dt^j  G^2  }_{2n-2j}^2+ \norm{\dt^j  u  }_{0}^2\nonumber
\\&\quad\ls   \norm{\dt^{j+1} u   }_{2n-2(j+1)+1 }^2 + \norm{\dt^{j }\eta  }_{2n-2 j +1 }^2+\bar{\mathcal{D}}_n^t+  \y_n. 
\end{align}
A simple induction on \eqref{ffes1} yields, since $\dt \eta=u$
\begin{align}\label{eses11} 
  \sum_{j=1}^{n}\norm{ \dt^j u  }_{2n-2j+1}^2+  \sum_{j=1}^{n-1}\norm{\nab \dt^j  p  }_{2n-2j-1}^2
 & \ls \ns{\dt^{n} u}_{1}+ \sum_{j=1}^{n-1}\norm{\dt^{j }\eta  }_{2n-2 j +1 }^2 +\bar{\mathcal{D}}_n^t+\y_n\nonumber
\\&   \ls   \sum_{j=1}^{n-1}\norm{\dt^{j-1} u   }_{2n-2(j-1)-1 }^2 + \bar{\mathcal{D}}_n^t+\y_n\nonumber
\\&  = \sum_{j=0}^{n-2}\norm{\dt^{j } u   }_{2n-2j-1 }^2   +\bar{\mathcal{D}}_n^t+ \y_n. 
\end{align}

{\bf Control terms without time derivatives}

Note that we can not use the Stokes problem \eqref{stokesp1} with $j=0$ as above since we have not controlled $\p_3^2\eta$ yet. But notice that we have certain control of the horizontal derivatives of $\eta$ in $\bar{\mathcal{D}}_n^\sharp$. This motivates us to introduce the quantity $w=  u+  \kappa\bar b^2/\mu \eta$ and we find that
\begin{equation}\label{stokesp2}
\begin{cases}
-\mu\Delta w+\nabla p=\kappa\bar b^2 \Delta_\ast\eta-\rho_0\partial_t u+G^1& \text{in }
\Omega
\\ \displaystyle\diverge w=  G^2+ \frac{ \kappa\bar b^2}{\mu}  \Phi& \text{in }
\Omega
\\  w=0 &\text{on }\p\Omega.
\end{cases}
\end{equation}
Fix $j=0,1,\dots,n-1$. Applying $\p^\al$ with $\al\in \mathbb{N}^2$ so that $|\al|\le 2n-2j-2$ to the problem \eqref{stokesp2} and then using the elliptic estimates of Lemma \ref{i_linear_elliptic2} with $2j+2\ge 2$ to obtain, summing over such $\al$,
\begin{align}\label{hihoh} 
&\norm{w}_{2j+2,2n-2j-2}^2+\norm{\nabla p}_{2j,2n-2j-2}^2\nonumber
\\&\quad\ls \norm{\Delta_\ast\eta}_{2j,2n-2j-2}^2+\norm{\p_t u}_{2j,2n-2j-2}^2+\norm{G^1}_{2j,2n-2j-2}^2+\norm{G^2}_{2j+1,2n-2j-2}^2\nonumber
\\&\qquad+\norm{\Phi}_{2j+1,2n-2j-2}^2+\ns{w}_{0,2n-2j} \nonumber
\\&\quad\ls \norm{ \eta}_{2j,2n-2j}^2+\norm{\p_t u}_{2n-2}^2+ \bar{\mathcal{D}}_n^t+\bar{\mathcal{D}}_n^\ast+\bar{\mathcal{D}}_n^\sharp  +\y_n. 
\end{align}
It is a key to note that
\begin{align} 
&\norm{w}_{2j+2,2n-2j-2}^2 = \norm{     u+     \frac{ \kappa\bar b^2}{\mu} \eta  }_{2j+2,2n-2j-2}^2\nonumber
\\&\quad= \norm{      u  }_{2j+2,2n-2j-2}^2+ \frac{ \kappa^2\bar b^4}{\mu^2}\norm{       \eta  }_{2j+2,2n-2j-2}^2+ \frac{ \kappa\bar b^2}{2\mu} \frac{d}{dt}\norm{     \eta  }_{2j+2,2n-2j-2}^2.
\end{align}
Therefore, we deduce that for $j=0,\dots, n-1$,
\begin{align} 
& \frac{d}{dt}\norm{     \eta  }_{2j+2,2n-2j-2}^2+\norm{      u  }_{2j+2,2n-2j-2}^2+\norm{        \eta  }_{2j+2,2n-2j-2}^2+\norm{\nabla p}_{2j,2n-2j-2}^2\nonumber
\\&\quad\ls \norm{ \eta}_{2j,2n-2j}^2+\norm{\p_t u}_{2n-2}^2  + \bar{\mathcal{D}}_n^t+\bar{\mathcal{D}}_n^\ast+\bar{\mathcal{D}}_n^\sharp +\y_n. 
\end{align}
By this recursive inequality on $j$, we conclude that there exist constants $\lambda_j>0,\ j=0,\dots, n-1$ such that
\begin{align}\label{q1345} 
& \frac{d}{dt}\sum_{j=0}^{n-1}\lambda_j\norm{     \eta  }_{2j+2,2n-2j-2}^2+\sum_{j=0}^{n-1}\left(\norm{      u  }_{2j+2,2n-2j-2}^2+\norm{        \eta  }_{2j+2,2n-2j-2}^2+\norm{\nabla p}_{2j,2n-2j-2}^2\right)\nonumber
\\&\quad\ls \norm{ \eta}_{0,2n}^2+\norm{\p_t u}_{2n-2}^2  + \bar{\mathcal{D}}_n^t+\bar{\mathcal{D}}_n^\ast+\bar{\mathcal{D}}_n^\sharp  +\y_n\nonumber
\\&\quad\ls  \norm{\p_t u}_{2n-2}^2  + \bar{\mathcal{D}}_n^t+\bar{\mathcal{D}}_n^\ast+\bar{\mathcal{D}}_n^\sharp  +\y_n. 
\end{align}
Hence if we define
\begin{equation}
\mathfrak{E}_n:=\sum_{j=0}^{n-1}\lambda_j\norm{     \eta  }_{2j+2,2n-2j-2}^2,
\end{equation}
then $\mathfrak{E}_n$ is equivalent to $\norm{\eta}_{2n}^2$. \eqref{q1345} implies in particular that
\begin{align} \label{eses22} 
& \frac{d}{dt}\mathfrak{E}_n
+ \norm{    u  }_{2n}^2+\norm{    \eta  }_{2n}^2+\norm{\nabla p}_{2n-2}^2 \ls  \norm{\p_t u}_{2n-2}^2 + \bar{\mathcal{D}}_n^t+\bar{\mathcal{D}}_n^\ast+\bar{\mathcal{D}}_n^\sharp   +\y_n. 
\end{align}

{\bf Combined estimates}

We may combine the estimates \eqref{eses11} and \eqref{eses22} to get
\begin{align}\label{ffes2} 
&
\frac{d}{dt}\mathfrak{E}_n
+   \norm{    u  }_{2n}^2+\norm{    \eta  }_{2n}^2+\norm{\nabla p}_{2n-2}^2
+ \sum_{j=1}^{n}\norm{ \dt^j u  }_{2n-2j+1}^2+  \sum_{j=1}^{n-1}\norm{\nab \dt^j  p  }_{2n-2j-1}^2\nonumber
  \\&\quad\ls  \sum_{j=0}^{n-2}\norm{\dt^{j } u   }_{2n-2j-1 }^2 +\norm{\p_t u}_{2n-2}^2 + \bar{\mathcal{D}}_n^t+\bar{\mathcal{D}}_n^\ast+\bar{\mathcal{D}}_n^\sharp +\y_n. 
\end{align}
Using the Sobolev interpolation and Young's inequality, we can improve \eqref{ffes2} to be
\begin{align}\label{claim2} 
&
\frac{d}{dt}\mathfrak{E}_n
+   \norm{    u  }_{2n}^2+\norm{    \eta  }_{2n}^2+\norm{\nabla p}_{2n-2}^2
 + \sum_{j=1}^{n}\norm{ \dt^j u  }_{2n-2j+1}^2+  \sum_{j=1}^{n-1}\norm{\nab \dt^j  p  }_{2n-2j-1}^2\nonumber
 \\&\quad \ls  \sum_{j=0}^{n-2}\norm{\dt^{j } u   }_{0}^2 + \bar{\mathcal{D}}_n^t+\bar{\mathcal{D}}_n^\ast+\bar{\mathcal{D}}_n^\sharp +\y_n\ls  \bar{\mathcal{D}}_n^t+\bar{\mathcal{D}}_n^\ast+\bar{\mathcal{D}}_n^\sharp   +\y_n. 
\end{align}
Adding $\bar{\mathcal{D}}_n^\ast$ to both sides of \eqref{claim2} implies that
\begin{equation}\label{dth_7}
\frac{d}{dt}\mathfrak{E}_n
+  \mathcal{D}_n   \ls  \bar{\mathcal{D}}_n^t+\bar{\mathcal{D}}_n^\ast+\bar{\mathcal{D}}_n^\sharp  +\y_n.
\end{equation}
Using \eqref{p_G_e_001} and \eqref{p_G_e_001''} to estimate
$\mathcal{Y}_{2N}\lesssim  { \se{N+2}  }(\sd{2N} +  \mathcal{J}_{2N} +\f)$, we obtain \eqref{d2n} from \eqref{dth_7} with $n=2N$; using   \eqref{p_G_e_002} and \eqref{p_G_e_002''}   to estimate $\mathcal{Y}_{N+2}\lesssim   {\mathcal{E}}_{2N}  {\mathcal{D}}_{N+2} $, we obtain \eqref{d2n+2} from \eqref{dth_7} with $n=N+2$.
\end{proof}

\subsubsection{Energy improvement}

Now we improve the energy estimates.

\begin{prop}\label{e2nic}
For $n=2N$ or $N+2$, it holds that
\begin{equation}\label{e2n}
{\mathcal{E}}_{n}  \lesssim
\bar{\mathcal{E}}_n^t   + \bar{\mathcal{E}}_n^\sharp   +\mathfrak{E}_{n}+ (\mathcal{E}_{n} )^{2}.
\end{equation}
\end{prop}
\begin{proof}
We let $n$ denote either $2N$ or $N+2$ throughout the proof, and we compactly write
\begin{equation}\label{n112}
 \mathcal{X}_n =   \ns{ \bar{\nab}^{2n-2}_0  G^1}_{0} +\ns{ \bar{\nab}^{2n-2}_0  G^2}_{1}  .
\end{equation}

For $j=0,\dots,n-1$, applying the elliptic estimates \eqref{stokes es} of Lemma \ref{i_linear_elliptic2} with $r=2n-2j \ge 2$ to the problem \eqref{stokesp1}, we obtain
\begin{align}\label{ffes3} 
 &\norm{\dt^j  u  }_{2n-2j}^2 + \norm{\nabla \dt^j  p  }_{2n-2j-2}^2\nonumber
 \\  &\quad\ls
\norm{\rho_0\dt^{j+1} u   }_{2n-2j-2 }^2+\norm{ \p_{3}^2 \dt^{j} \eta}_{2n-2j-2 }^2+ \norm{ \dt^j G^1   }_{2n-2j-2}^2
+ \norm{\dt^j  G^2  }_{2n-2j-1}^2+ \norm{\dt^j  u  }_{0}^2\nonumber
  \\&\quad\ls   \norm{\dt^{j+1} u   }_{2n-2(j+1) }^2 +\norm{   \dt^{j} \eta}_{2n-2j  }^2  +\bar{\mathcal{E}}_n^t +  \x_n.
\end{align}
A simple induction on \eqref{ffes3} yields, since $\dt\eta=u$,
\begin{align}\label{ffes4} 
\sum_{j=0}^{n}\norm{ \dt^j u  }_{2n-2j }^2+ \norm{\nabla \dt^j  p  }_{2n-2j-2}^2\nonumber
 &  \ls \ns{\dt^{n} u}_{0} +\sum_{j=0}^{n-1}\norm{ \dt^j \eta  }_{2n-2j }^2 +\bar{\mathcal{E}}_n^t +  \x_n
\\&   \le \bar{\mathcal{E}}_n^t +\norm{  \eta  }_{2n  }^2 +\sum_{j=1}^{n-1}\norm{ \dt^{j-1} u  }_{2n-2j }^2   + \x_n\nonumber
\\&  \ls  \sum_{j=0}^{n-2}\norm{ \dt^{j} u  }_{2n-2j-2}^2     +\bar{\mathcal{E}}_n^t  +\mathfrak{E}_n +  \x_n. 
\end{align}
Using the Sobolev interpolation and Young's inequality, we can improve \eqref{ffes4} to be
\begin{align}\label{claim2'} 
&  \sum_{j=0}^{n}\norm{ \dt^j u  }_{2n-2j }^2+   \norm{\nabla \dt^j  p  }_{2n-2j-2}^2\nonumber
\\&\quad \ls  \sum_{j=0}^{n-2}\norm{ \dt^{j} u  }_{0}^2   + \bar{\mathcal{E}}_n^t +\mathfrak{E}_n +  \x_n
\ls   \bar{\mathcal{E}}_n^t  +\mathfrak{E}_n +  \x_n. 
\end{align}
Adding $\bar{\mathcal{E}}_n^\sharp$ to both sides of \eqref{claim2'} implies that
\begin{equation}\label{claim12}
{\mathcal{E}}_{n}  \lesssim
 \bar{\mathcal{E}}_{n}^t   +\bar{\mathcal{E}}_n^\sharp+ \mathfrak{E}_{n} +  \x_n.
\end{equation}
Using \eqref{p_G_e_0} to bound $\mathcal{X}_{n} \lesssim  (\mathcal{E}_{n} )^{2}$, we then conclude \eqref{e2n}.
\end{proof}

\subsubsection{Synthesis}

We now chain all the estimates derived previously to conclude the following.
\begin{prop}
For $n=N+2$ or $2N$, there exists an  energy $\tilde{ \mathcal{E}}_{n}$ which is equivalent to $\mathcal{E}_{n}$ such that
\begin{equation}\label{sys2n}
\frac{d}{dt}\tilde{ \mathcal{E}}_{2N}+ {\mathcal{D}}_{2N}  \ls     \sqrt{ \se{N+2}  }(  \mathcal{J}_{2N} +\f)
\end{equation}
and
\begin{equation}\label{sysn+2}
\frac{d}{dt}\tilde{ \mathcal{E}}_{N+2}+ {\mathcal{D}}_{N+2}  \le 0.
\end{equation}
\end{prop}
\begin{proof}
We let $n$ denote either $2N$ or $N+2$ through the proof, and we use the compact notation
\begin{equation}
\mathcal{Z}_n \text{ with }\mathcal{Z}_{2N}:=\sqrt{ \se{N+2}  }(\sd{ 2N}+  \mathcal{J}_{2N} +\f)\text{ and }\mathcal{Z}_{N+2}:=\sqrt{ \se{2N}  }\sd{ N+2} .
\end{equation}
We then deduce from Propositions \ref{i_temporal_evolution  N}, \ref{p_upper_evolution  N12}, \ref{p_upper_evolution  N'132} and \ref{p_upper_evolution  N'} that for $K\gg1$ and $0<\epsilon\ll 1$,
\begin{align} \label{i_te_02n} 
 & \frac{d}{dt}\left(K\bar{\mathcal{E}}_{n}^t+K\int_\Omega  \nabla \dt^{n-1} p   Q^{2,n}+ \bar{\mathcal{E}}_{n}^*+\epsilon\left( \bar{\mathcal{E}}_{n}^\sharp+2\sum_{\substack{\al\in \mathbb{N}^2\\  |\al|\le 2n}}\int_\Omega \partial^\alpha ( \rho_0   u )\cdot \partial^\alpha  \eta\right)+\epsilon^2 \mathfrak{E}_{n}\right)\nonumber
 \\&\quad + K\bar{\mathcal{D}}_{n}^t+\bar{\mathcal{D}}_{n}^*+\epsilon \bar{\mathcal{D}}_{n}^\sharp+\epsilon^2  {\mathcal{D}}_{n}\nonumber
 \\&\quad \ls   K \mathcal{Z}_n+\norm{\dt u}_{2n-2}^2+\epsilon(\bar{\mathcal{D}}_{n}^t+\bar{\mathcal{D}}_{n}^{\ast}+\norm{  u}_{2n-1}^2)+\epsilon^2(\bar{\mathcal{D}}_{n}^t+\bar{\mathcal{D}}_{n}^*+ \bar{\mathcal{D}}_{n}^\sharp). 
\end{align}
Taking $\epsilon>0$ sufficiently small, we obtain
\begin{align} \label{i_te_02n1} 
 & \frac{d}{dt}\left(K\bar{\mathcal{E}}_{n}^t+K\int_\Omega  \nabla \dt^{n-1} p   Q^{2,n}+ \bar{\mathcal{E}}_{n}^*+\epsilon\left( \bar{\mathcal{E}}_{n}^\sharp+2\sum_{\substack{\al\in \mathbb{N}^2\\  |\al|\le 2n}}\int_\Omega \partial^\alpha ( \rho_0   u )\cdot \partial^\alpha  \eta\right)+\epsilon^2 \mathfrak{E}_{n}\right)\nonumber
 \\&\quad + K\bar{\mathcal{D}}_{n}^t+\bar{\mathcal{D}}_{n}^*+\epsilon \bar{\mathcal{D}}_{n}^\sharp+\epsilon^2  {\mathcal{D}}_{n}\nonumber\\&\quad \ls    K \mathcal{Z}_n+\norm{\dt u}_{2n-2}^2+ \norm{  u}_{2n-1}^2 .
\end{align}
By the Sobolev interpolation and Young's inequality, we have
\begin{equation} \label{epsionl}
\norm{\dt u}_{4N-2}^2 +\norm{  u}_{4N-1}^2 \ls \frac{1}{\epsilon^3}(\norm{\dt u}_{0}^2 +\norm{  u}_{0}^2)+{\epsilon^3}(\norm{\dt u}_{4N-1}^2 +\norm{  u}_{4N}^2)
 \ls \frac{1}{\epsilon^3}\bar{\mathcal{D}}_{n}^t +{\epsilon^3}{\mathcal{D}}_{n}.
\end{equation}
Plugging \eqref{epsionl} into \eqref{i_te_02n1} and taking $K=1/\epsilon^4$, we deduce that for sufficiently small $\epsilon>0$,
\begin{align} \label{i_te_02n2} 
 & \frac{d}{dt}\left(\frac{1}{\epsilon^4}\bar{\mathcal{E}}_{n}^t +\frac{1}{\epsilon^4}\int_\Omega  \nabla \dt^{n-1} p   Q^{2,n}+ \bar{\mathcal{E}}_{n}^*+\epsilon\left( \bar{\mathcal{E}}_{n}^\sharp+2\sum_{\substack{\al\in \mathbb{N}^2\\  |\al|\le 2n}}\int_\Omega \partial^\alpha ( \rho_0   u )\cdot \partial^\alpha  \eta\right)+\epsilon^2 \mathfrak{E}_{n}\right)\nonumber
 \\&\quad + K\bar{\mathcal{D}}_{n}^t+\bar{\mathcal{D}}_{n}^*+\epsilon \bar{\mathcal{D}}_{n}^\sharp+\epsilon^2  {\mathcal{D}}_{n} \ls \frac{1}{\epsilon^4}\mathcal{Z}_n. 
\end{align}

We now define
\begin{equation}
\tilde{ \mathcal{E}}_{n}:=\frac{1}{\epsilon^4}\bar{\mathcal{E}}_{n}^t +\frac{1}{\epsilon^4}\int_\Omega  \nabla \dt^{n-1} p   Q^{2,n}+ \bar{\mathcal{E}}_{n}^*+\epsilon\left( \bar{\mathcal{E}}_{n}^\sharp+2\sum_{\substack{\al\in \mathbb{N}^2\\  |\al|\le 2n}}\int_\Omega \partial^\alpha ( \rho_0   u )\cdot \partial^\alpha  \eta\right)+\epsilon^2 \mathfrak{E}_{n} .
\end{equation}
By \eqref{p_F_e_02},
\begin{equation}
\int_\Omega  \nabla \dt^{n-1} p   Q^{2,n}\ls \sqrt{\mathcal{E}_n}\sqrt{\mathcal{E}_{N+2}\mathcal{E}_n}=\sqrt{\mathcal{E}_{N+2}}\mathcal{E}_n,
\end{equation}
together with Proposition \ref{e2nic}, we know that for fixed sufficiently small $\epsilon>0$,
\begin{equation}
\mathcal{E}_{n}\ls \tilde{ \mathcal{E}}_{n}+(\mathcal{E}_{n} )^{2},
\end{equation}
which implies that $\tilde{ \mathcal{E}}_{n}$ is equivalent to $ \mathcal{E}_{n}$ since ${\mathcal{E}}_{2N} (T)\le  \delta$ is small. We thus deduce \eqref{sys2n} and \eqref{sysn+2} from \eqref{i_te_02n2} by recalling the notation $\mathcal{Z}_n$ and using again that ${\mathcal{E}}_{2N} (T)\le \delta$ is small.
\end{proof}

\subsection{Global energy estimates}

In this subsection, we shall conclude our global energy estimates of the solution to \eqref{reformulationic}.

We begin with the estimate of $\f $ and $\mathcal{J}_{2N}$.

\begin{prop}\label{p_f_bound}
There exists a universal constant $0<\delta<1$ so that if $\mathcal{G}_{2N}(T)\le\delta$, then
\begin{equation}\label{grow1}
\f(t)
  \ls  \f(0)+  \sup_{0\le r\le t}\mathcal{E}_{2N}(r)+ \int_0^t \mathcal{D}_{2N} \text{ for all
}0\le t\le T
\end{equation}
and for any $\vartheta>0$,
\begin{equation}\label{grow2}
  \int_0^t \frac{\f+\mathcal{J}_{2N}}{(1+r)^{1+\vartheta}}dr
  \ls  \f(0)+  \sup_{0\le r\le t}\mathcal{E}_{2N}(r)+ \int_0^t \mathcal{D}_{2N} \text{ for all
}0\le t\le T.
\end{equation}
\end{prop}
\begin{proof}
Following the arguments lead to \eqref{eses22} (basically, start with replacing $2j+2$ with $2j+3$ in \eqref{hihoh}), we deduce that there exists an energy $\tilde{\mathcal{F}}_{2N}$ which is equivalent to $\f$ such that
\begin{equation}
 \dtt\tilde{\mathcal{F}}_{2N}
+  \f +\mathcal{J}_{2N}  \ls  \norm{\eta}_{1,4N}^2+ \norm{\p_t u}_{4N-1}^2
 + \norm{G^1}_{4N-1}^2+\norm{G^2}_{4N}^2+\norm{\Phi}_{4N}^2.
\end{equation}
We use \eqref{p_G_e_001} and \eqref{p_G_e_001''} to estimate
\begin{equation}
\norm{G^1}_{4N-1}^2+\norm{G^2}_{4N}^2+\norm{\Phi}_{4N}^2\lesssim  { \se{N+2}  }(\sd{2N} +  \mathcal{J}_{2N} +\f).
\end{equation}
Then we have
\begin{equation}\label{109}
 \dtt\tilde{\mathcal{F}}_{2N}
+  \f +\mathcal{J}_{2N}  \ls  \se{2N}+ \sd{2N}
 + { \se{N+2}  } \sd{2N}\ls  \se{2N}+ \sd{2N}.
\end{equation}
since $\se{N+2}(t) \le \delta$ is small.

We now employ the time weighted analysis on \eqref{109}. First, a Gronwall type analysis on \eqref{109} yields
\begin{align}\label{1001} 
  \f  & \ls  \f(0) e^{-t}+\int_0^t e^{-(t-r)}\left(  \se{2N}(r)+ \sd{2N}(r) \right)dr\nonumber
  \\& \ls \f(0) e^{-t}+\sup_{0\le r\le t}\mathcal{E}_{2N}(r) \int_0^t e^{-(t-r)}dr+  \int_0^t \sd{2N}(r)  dr, 
\end{align}
which in particular yields \eqref{grow1}.

On the other hand, multiplying \eqref{109} by $(1+t)^{-1-\vartheta}$ for any $\vartheta>0$, we obtain
\begin{equation}\label{1110}
 \dtt \left( \frac{\tilde{\mathcal{F}}_{2N}}{(1+t)^{1+\vartheta}}\right)+(1+\vartheta)\frac{\tilde{\mathcal{F}}_{2N}}{(1+t)^{2+\vartheta}}
+  \frac{\f+\mathcal{J}_{2N}}{(1+t)^{1+\vartheta}} \ls   \frac{\se{2N}}{(1+t)^{1+\vartheta}}+  \frac{\sd{2N}}{(1+t)^{1+\vartheta}} .
\end{equation}
 Integrating \eqref{1110} directly in time yields \eqref{grow2}.
\end{proof}

 Now we show the boundedness of $\mathcal{E}_{2N} +\int_0^t\mathcal{D}_{2N} $.

\begin{prop} \label{Dgle}
There exists a universal constant $0<\delta<1$ so that if $\mathcal{G}_{2N}(T)\le\delta$, then
\begin{equation}\label{Dg}
\mathcal{E}_{2N} (t)+\int_0^t\mathcal{D}_{2N} \lesssim
\mathcal{E}_{2N} (0) + \mathcal{F}_{2N}(0)  \text{ for all
}0\le t\le T.
\end{equation}
\end{prop}
\begin{proof}
Integrating \eqref{sys2n} directly in time, we find that
\begin{equation}
  {\mathcal{E}}_{2N}
+ \int_0^t{\mathcal{D}}_{2N}
\ls {\mathcal{E}}_{2N}(0)+\int_0^t \sqrt{ \se{N+2}  }(  \mathcal{J}_{2N} +\f)
\end{equation}
By the estimates \eqref{grow2} of Proposition \ref{p_f_bound}, we deduce
\begin{align} 
  {\mathcal{E}}_{2N}
+ \int_0^t{\mathcal{D}}_{2N}
&\ls {\mathcal{E}}_{2N}(0)+ \int_0^t \sqrt{\delta}(1+r)^{-N+2}(  \mathcal{J}_{2N} +\f)dr\nonumber
\\& \ls {\mathcal{E}}_{2N}(0)+ \sqrt{ \delta } \left(\f(0)+  \sup_{0\le r\le t}\mathcal{E}_{2N}(r)+ \int_0^t \mathcal{D}_{2N}\right)
. 
\end{align}
Here we have used the fact that $N-2\ge 1+\vartheta$. This proves the estimate \eqref{Dg} since $\delta$ is small.
\end{proof}

It remains to show the decay estimates of
$\mathcal{E}_{N+2}$.

\begin{prop} \label{decaylm}
There exists a universal constant $0<\delta<1$ so that if $\mathcal{G}_{2N}(T)\le\delta$, then
\begin{equation}\label{n+2}
(1+t)^{2N-4} \mathcal{E}_{N+2} (t)\lesssim
\mathcal{E}_{2N} (0)+ \mathcal{F}_{2N}(0) \ \text{for all
}0\le t\le T.
\end{equation}
\end{prop}
\begin{proof}
The proposition follows essentially in the same way as Proposition \ref{decaylmc}.
\end{proof}

Now we can arrive at our ultimate energy estimates for
$\mathcal{G}_{2N} $.
\begin{thm}\label{Ap}
There exists a universal $0 < \delta < 1$ so that if $
\mathcal{G}_{2N} (T) \le \delta$, then
\begin{equation}\label{Apriori}
 \mathcal{G}_{2N} (t) \ls\mathcal{E}_{2N} (0)+ \mathcal{F}_{2N}(0) \text{ for all }0 \le t \le
 T.
\end{equation}
\end{thm}
\begin{proof}
The conclusion follows directly from the definition of
$\mathcal{G}_{2N} $ and Propositions
\ref{p_f_bound}--\ref{decaylm}.
\end{proof}

\appendix

\section{Elliptic regularity}\label{section_appendix}

We first recall the classical regularity theory for the Lam$\acute{e}$ system:
\begin{equation}\label{lame eq}
\begin{cases}
-\mu\Delta u -(\mu+\mu')\nabla \diverge u =f  \quad &\hbox{in }\Omega
\\u=0\quad &\hbox{on }\pa \Omega.
\end{cases}
\end{equation}
\begin{lemma}\label{i_linear_lame}
Let $r\ge 2$. If $f\in H^{r-2}(\Omega)$, then there exists unique $u\in H^r(\Omega)$ solving \eqref{lame eq}. Moreover,
\begin{equation}
\norm{u}_{r}\lesssim\norm{f}_{r-2}.
\end{equation}
\end{lemma}
\begin{proof}
 See \cite{ADN}.
\end{proof}

We next recall the classical regularity theory for the Stokes system:
\begin{equation}\label{stokes eq}
\begin{cases}
-\mu\Delta u +\nabla p =f  \quad &\hbox{in }\Omega
\\\diverge{u} =g  \quad  &\hbox{in }\Omega
\\u=0\quad &\hbox{on }\pa \Omega.
\end{cases}
\end{equation}
\begin{lemma}\label{i_linear_elliptic2}
Let $r\ge 2$. If $f\in H^{r-2}(\Omega)$, $g\in H^{r-1}(\Omega)$
and  $(u,p) $ solves \eqref{stokes eq}, then
\begin{equation}\label{stokes es}
\norm{u}_{r}+\norm{\nabla p}_{r-2}\lesssim\norm{f}_{r-2}+\norm{g}_{r-1}+\norm{u}_0.
\end{equation}
\end{lemma}
\begin{proof}
See \cite{L,T}.
\end{proof}%
\begin{remark}
	Note that  to guarantee the existence of the unique solution to  \eqref{stokes eq} as stated in Lemma \ref{i_linear_elliptic2}, we may need to impose the following structure condition of $g$:
	\begin{equation}\label{structure}
	g=\diverge \varphi\text{ with }\varphi\in L^2(\Omega),\ \varphi_3=0\text{ on }\pa\Omega.
	\end{equation}
	Moreover, $\norm{u}_0$ in the right hand side of the estimate \eqref{stokes es} can then be replaced by $\norm{\varphi}_0$. Although the Stokes problems employed in this paper satisfy the structure condition \eqref{structure}, we do not pursue such sort of estimates since we have already controlled $\norm{u}_0$ in our applications.
\end{remark}

 \section*{Acknowledgements}

The authors are deeply grateful to the referees for the invaluable comments and suggestions.

    \end{document}